\documentclass[final,onefignum,onetabnum]{siamart190516} 

\usepackage{lipsum}
\usepackage{amsfonts}
\usepackage{graphicx}
\usepackage{epstopdf}
\ifpdf
  \DeclareGraphicsExtensions{.eps,.pdf,.png,.jpg}
\else
  \DeclareGraphicsExtensions{.eps}
\fi

\newsiamremark{remark}{Remark}
\newsiamremark{hypothesis}{Hypothesis}
\crefname{hypothesis}{Hypothesis}{Hypotheses}
\newsiamthm{claim}{Claim}

\headers{Efficient Numerical Algorithms for the GLE}{B. Leimkuhler, M. Sachs}

\title{Efficient Numerical Algorithms for the Generalized Langevin Equation\thanks{Submitted to the editors DATE.
\funding{This work was made possible by the RULE project of the European Research Council, grant no. 320823.}}}

\author{
Benedict Leimkuhler\thanks{School of Mathematics, University of Edinburgh, UK
  (\email{b.leimkuhler@ed.ac.uk}, \url{http://kac.maths.ed.ac.uk/\string~bl/}).}
\and Matthias Sachs \thanks{Department of Mathematics, University of British Columbia, Canada
  (\email{msachs@math.ubc.ca}, \url{https://sites.google.com/view/matthiassachs}).}
}

\usepackage{amsopn}
\DeclareMathOperator{\diag}{diag}

\usepackage{multirow}
  \usepackage{tablefootnote}

\usepackage{mathtools,xparse}
\usepackage{mathrsfs}
\usepackage{enumitem} 
\usepackage{placeins}
\usepackage{accents}
\usepackage{bm}

\def\LcLD{{\Lc}_{\rm LD}}
\def\Lca{{\Lc}_{\rm A}}
\def\Lcb{{\Lc}_{\rm B}}
\def\Lco{{\Lc}_{\rm O}}
\def\Lch{{\Lc}_{\rm H}}
\def\Lcx{{\Lc}_{\rm X}}
\def\Lcy{{\Lc}_{\rm Y}}
\def\Lcz{{\Lc}_{\rm Z}}
\def\LcLDo{{\Lc}_{\widehat{{\rm O}}}}

\newcommand\given[1][]{\:#1\vert\:}

\def\EE{{ \mathbb{E}}}
\def\RR{{\mathbb R}}

\def\error{\mathcal{E}}
\DeclarePairedDelimiter{\innerLmu}{\langle}{\rangle_{L^{2}(\muinv)}}

\def\ccdot{\;\cdot\;}
\def\G{\bm G}
\def\Ed{\widehat{{\bm E}}_{\deltat}}
\def\Ec{{\bm E}_{\deltat}}

\def\GammaLD{\widehat{\Gammabf}}
\def\Ug{{U_{1}}}
\def\Ub{{U_{2}}}
\def\trans{{T}}
\def\NN{{\mathbb{N}}}

\DeclarePairedDelimiter{\floor}{\lfloor}{\rfloor}
\def\({\left (}
\def\){\right )}

\def\xDomain{{\Omega_{\x}}}

\def\qDomain{{\Omega_{\q}}}
\def\pDomain{{\Omega_{\p}}}
\def\sDomain{{\Omega_{\s}}}

\def\testfunc{\mathcal{C}_{P}^{\infty}(\xDomain,\RR)}

\def\muinv{\pi}
\def\bigO{{O}}

\def\dq{{\widehat{\q}}}

\def\dz{{\widehat{\z}}}
\def\dx{{\widehat{\x}}}

\def\lq{{\overline{\q}}}

\def\lz{{\overline{\z}}}
\def\lx{{\overline{\x}}}

\def\xnum{\widehat{\x}}

\def\Rand{{\mathscr R}}

\def\EE{{ \mathbb{E}}}
\def\RR{{\mathbb R}}

\def\K{{ \bf K }}

\def\deltat{{h }}

\DeclarePairedDelimiter{\abs}{\lvert}{\rvert}
\DeclarePairedDelimiter{\norm}{\lVert}{\rVert}
\NewDocumentCommand{\normLinfty}{ s O{} m }{%
  \IfBooleanTF{#1}{\norm*{#3}}{\norm[#2]{#3}}_{L_\infty}%
}
\NewDocumentCommand{\normLtwo}{ s O{} m }{%
  \IfBooleanTF{#1}{\norm*{#3}}{\norm[#2]{#3}}_{L_2}%
}
\NewDocumentCommand{\normEuc}{ s O{} m }{%
  \IfBooleanTF{#1}{\norm*{#3}}{\norm[#2]{#3}}_{2}%
}
\NewDocumentCommand{\normFrob}{ s O{} m }{%
  \IfBooleanTF{#1}{\norm*{#3}}{\norm[#2]{#3}}_{{\rm F}}%
}

\def\bf{\bm}
\def\W{{\bf W}}

\def\I{{\bf I}}
\def\S{{{\bf S}_{\deltat}}}
\def\Sw{\widetilde{\S}}
\def\Qw{\widetilde{\Q}}

\def\F{{ \bf F}}
\def\G{{\bf G}}

\def\dd{{\rm d}}

\def\q{{ \bf q}}

\def\M{{ \bf M}}
\newcommand{\etabf}{{\boldsymbol \eta}}
\def\p{{ \bf p}}
\def\s{{ \bf s}}
\def\x{{ \bf x}}
\def\z{{ \bf z}}
\def\F{{{ \bf F}_{\deltat}}}
\def\Q{{ \bf Q}}
\def\C{{ \bf C}}
\def\0{{ \bf 0}}

\def\Lc{\mathcal{L}}
\def\Pc{\mathcal{P}}
\def\Kc{\mathcal{K}}
\def\Acal{\mathcal{A}}

\def\Lcg{\mathcal{\Lc}_{\rm GLE}}

\def\Kbf{{\bf K}}

\def\thetabf{{  \bm{\theta}}} 
 
\def\Gammabf{{  \bm{\Gamma} }} 
\def\Sigmabf{{  \bm{\Sigma} }}

\def\mubf{{  \bm{\mu} }} 
\def\Omegabf{{  \bm{\Omega} }} 

\def\Psibf{{\bm{\Psi}}} 
\def\numin{\hat{\Phi}}
\def\numint{\numin_{\deltat}}

\usepackage{setspace}
\usepackage[algo2e,ruled,vlined,commentsnumbered]{algorithm2e}
\usepackage{algorithmic}
\usepackage[absolute,overlay]{textpos}
\usepackage{tabularx}
\usepackage{tikz}
\usepackage{textcomp}
\usepackage{graphicx}
\usepackage{mathtools}
\usepackage[many]{tcolorbox}
\usetikzlibrary{positioning}
\newtcolorbox{eqtbox}[1][]{%
    enhanced,breakable,
    colframe=black,
    colback=white,
    arc=1mm,
    borderline={0.1mm}{0mm}{white},
    attach boxed title to top center={yshift=-3mm},
    boxed title style={size=small,colback=white,boxsep=1.0mm,left=0mm,right=0mm}, 
    coltitle=black,
    boxsep=1mm,
    left=0mm,
    right=0mm,
    before={\noindent},
    segmentation style={thin, red!30},
    title=#1%
}

\usepackage[nolist]{acronym}
\begin{acronym}
\acro{GLE}{Generalized Langevin equation}
\end{acronym}
\ifpdf
\hypersetup{
  pdftitle={Efficient Numerical Algorithms for the Generalized Langevin Equation},
  pdfauthor={M. Sachs, B. Leimkuhler}
}
\fi

\usepackage{bbm, dsfont}

\newtheorem{assumption}{Assumption}

\usepackage{cleveref}
\crefname{thm}{Theorem}{Theorems}
\crefname{assumption}{Assumption}{Assumptions}

\crefformat{equation}{(#2#1#3)}
\crefname{theorem}{Theorem}{Theorems}
\crefname{proposition}{Proposition}{Propositions}
\crefname{lemma}{Lemma}{Lemmas}
\crefname{collorary}{Collorary}{Colloraries}
\crefname{assumption}{Assumption}{Assumptions}
\crefname{remark}{Remark}{Remark}
\crefname{example}{Example}{Example}
\crefname{definition}{Definition}{Definitions}
\crefname{section}{Section}{Sections}
\crefname{figure}{Figure}{Figures}
\crefname{chapter}{Chapter}{Chapters}
\crefname{table}{Table}{Tables}
\crefname{algorithm2e}{Algorithm}{Algorithms}

\newcommand{\mscom}[1]{\textcolor{magenta}{MS: #1}}

\newcounter{mycounter}

\usepackage{newfloat}
\DeclareFloatingEnvironment[
    fileext=los,
    listname={List of Schemes},
    name=Algorithms,
    placement=tbhp,
]{scheme}

\begin{document}

\maketitle

\begin{abstract}
We 
study the design and implementation of numerical methods to solve the generalized Langevin equation (GLE) focusing on canonical sampling properties of numerical integrators. For this purpose, we cast the GLE in an extended phase space formulation and derive a family of splitting methods which generalize existing Langevin dynamics integration methods. We show exponential convergence in law and the validity of a central limit theorem for the Markov chains obtained via these integration methods, and we show that the dynamics of a suggested integration scheme is consistent with asymptotic limits of the exact dynamics and can reproduce (in the short memory limit) a superconvergence property for the analogous splitting of underdamped Langevin dynamics.   We then apply our proposed integration method to several model systems, including a Bayesian inference problem.  We demonstrate in numerical experiments that
our method outperforms other proposed GLE integration schemes in terms of the accuracy of  sampling.  Moreover, using a parameterization of the memory kernel in the GLE as proposed by Ceriotti et al \cite{Ceriotti2010}, our experiments indicate that the obtained GLE-based sampling scheme outperforms state-of-the-art sampling schemes based on underdamped Langevin dynamics in terms of robustness and efficiency.
\end{abstract}

\begin{keywords}
Generalized Langevin dynamics, approximate Markov chain Monte Carlo, symmetric splitting, BAOAB
\end{keywords}

\begin{AMS}
65C30, 65C40, 82M37
\end{AMS}

\section{Introduction}\label{sec:introduction}
In this article we study numerical discretization schemes for a generalized Langevin equation (GLE) of the form
\begin{equation}\label{eq:GLE:nonmark}\tag{GLE}
\begin{aligned}
\dot{\q} &=  \M^{-1}\p, \\
\dot {\p} &=  -\nabla_{\q}U(\q) -  \int_{0}^{t}\K(t-s)\M^{-1}\p(s) \dd s  + \etabf(t).
\end{aligned}
\end{equation}
where the dynamic variables $\q \in \qDomain = \RR^{n}, \p \in \pDomain =\RR^{n}$ denote the position and momenta of a  Hamiltonian system with energy function 
\begin{equation}\label{eq:hamiltonian:gen}
H(\q,\p) = U(\q) + \frac{1}{2} \p^{\trans}\M^{-1}\p.
\end{equation}
 The mass matrix $\M \in \mathbb{R}^{n\times n}$ is assumed to be symmetric positive definite, and $U$ is a smooth confining potential function (i.e., $U\in \mathcal{C}^{\infty}(\qDomain, \RR)$ and  $U(\q) \rightarrow \infty$ as $\norm{\q} \rightarrow \infty$) such that $\int_{\qDomain} {\rm e}^{-U(\q)} \dd \q <\infty$. $\K : [0,\infty) \rightarrow \mathbb{R}^{n\times n }$ is a matrix-valued (generalized) function, which is referred to as the memory kernel, and $\etabf $ is a stationary Gaussian process with vanishing mean taking values in $\mathbb{R}^{n}$ and which (in equilibrium) is assumed to be statistically independent of $\q$ and $\p$. The dynamical variable $\etabf$ models a random force, which is such that a fluctuation dissipation relation holds. That is, the  auto-covariance function of the random force and the memory kernel $\K$ coincide up to a constant pre-factor, i.e., 
 \[
 \EE[\etabf(s+t) \etabf^{\top}(s)] = \beta^{-1} \K(t), \text{ for all } t,s>0.
 \]
 In its general form \eqref{eq:GLE:nonmark} the GLE is a non-Markovian dynamical model, meaning that the evolution of the state of the described system depends not only on the state itself but on the state history. The {\em underdamped Langevin equation} is obtained as a special Markovian variant of the GLE when the memory kernel is chosen as $\K(t)=\GammaLD\delta(t)$ and $\eta = \sqrt{2\GammaLD \beta^{-1}}\dot{\W}$, where $\delta(\cdot)$ denotes the Dirac delta function,
 $\dot{\W}$ is a Gaussian white noise in $\RR^{n}$ with independent components, i.e., $\dot{\W}=[\dot{W}_{i}]_{1\leq i \leq n}$, such that, $\dot{W}_{i}\sim \mathcal{N}(0,1)$ and $\mathbb{E}[\dot{W}_{i}(t) \dot{W}_{j}(s)] = \delta_{ij} \delta(t-s)$,
  and $\GammaLD$ is a symmetric positive  definite matrix which is commonly referred to as the friction matrix. Under this parameterization \eqref{eq:GLE:nonmark} simplifies to the It\^o diffusion 
\begin{equation}\label{eq:LD}\tag{LD}
\dot{\q} =  \M^{-1}\p, ~~~\dot {\p} =  -\nabla_{\q}U(\q) -  \GammaLD\M^{-1} \p  + \sqrt{2\GammaLD\beta^{-1}} \dot{\W}.
\end{equation}

\subsection{The GLE as a dynamical model}\label{sec:int:dyn:model}
Traditionally, the generalized Langevin equation is widely used in thermodynamics to model the dynamics of an open system which exchanges energy with one or more heat baths. The equation can be formally derived via the Mori-Zwanzig projection formalism \cite{Mori1965,Zwanzig1973}. As such, it provides a dynamical description of the projection of a physical system onto a finite subset of its degrees of freedom. In the absence of a clear scale separation in the time evolution of explicitly modeled degrees of freedom and the traverse degrees of freedom, Markovian approximations in the form of \eqref{eq:LD} fail to reproduce the dynamical properties of the system. Incorporation of memory effects via the stochastic integro-differential equation  \eqref{eq:GLE:nonmark} are key to an accurate description of the system dynamics in such a setup. As such the GLE is used as a dynamical model in a wide range of applications, including coarse grained meso- and macro-scale molecular particle dynamics models \cite{Givon2004,Li2015a,li2017computing}, simulation of solids \cite{Kantorovich2008,Ness2015}, non-equilibrium dynamics in open systems with temperature gradient \cite{Eckmann1998,Rey-Bellet2002,Ness2016}, 
complex fluids and (anomalous) diffusive transport in soft matter \cite{Fricks2009a,McKinley2009,vasquez2015complex}.

\subsection{Application in sampling}\label{sec:int:sampling}
Besides its application as a dynamical model, the GLE has been used in molecular sampling to design (approximate) Markov chain Monte Carlo (MCMC) methods with enhanced sampling properties \cite{Ceriotti2009,Ceriotti2010,Morrone2011,Wu2015,gle4md}, and there are theoretical results showing that GLE-based sampling schemes \cite{mou2019high} and annealing schemes \cite{chak2020generalised} can exhibit better convergence properties in comparison to schemes based on an  (underdamped) Langevin equation. Indeed, under certain conditions  (see \cref{sec:ergodicity}) on the memory kernel $\K$, the process defined by \eqref{eq:GLE:nonmark} is exponentially ergodic with unique invariant measure given by the Gibbs-Boltzmann distribution
\begin{equation}\label{eq:GB}
\muinv(\dd \q\,\dd\p) = \frac{1}{Z} {\rm e}^{-\beta H(\q,\p) }\dd \q \dd \p,
\end{equation}
so that in particular, the process $\q(t), t \geq 0$ can be used to draw samples from the marginal measure 
\begin{equation}
\muinv_{q} (\dd \q ) \propto {\rm e}^{-\beta U(\q)} \dd \q. 
\end{equation}
In a nutshell, the idea behind the enhanced sampling methods developed in \cite{Ceriotti2009,Ceriotti2010} is to equip the GLE with  a memory kernel of the form $\K(t) =K(t) \I_{n}$ which is constructed such that mixing times associated to \eqref{eq:GLE:nonmark} with $U(\q) =  \frac{1}{2}\omega \q^{2} $ are minimized over a prescribed frequency range $\omega \in [\omega_{\min},\omega_{\max}] \subset (0,\infty)$. The resulting memory kernel  is such that the effect of the convolution $\int_{0}^{t}\K(t-s)\M^{-1}\p(s) \dd s$  
behaves like a high-pass filter on the momentum trajectory. This allows slower modes to evolve almost ballistically while faster modes are sufficiently damped  to avoid resonance effects in numerical discretizations of the dynamics. Even though this construction assumes the target $\muinv_{q}$ to be Gaussian, the enhanced sampling properties have been shown in practice to extend to the non-Gaussian case \cite{Ceriotti2010,Ceriotti2010b}. In particular in sampling problems where $\muinv_{q}$ is ill conditioned (in the sense that the associated covariance matrix has a large condition number), such constructed sampling schemes may result in drastically improved sampling efficiency and robustness in comparison to schemes obtained by discretization of an underdamped Langevin equation \cite{Ceriotti2010,tayu2020}.

\subsection{Scope and main results of this article}
The purpose of this article is to provide a class of numerical integrators with well understood theoretical properties 
which are suitable for simulation of the GLE in a wide range of applications. 
Within this class of numerical integrators we identify one scheme, \textnormal{gle-BAOAB}, which we show analytically and in numerical experiments to outperform previously proposed schemes in terms of numerical discretization error and numerical stability. 

The fact that in any computer simulation only finite memory is available means that any computer simulation of \eqref{eq:GLE:nonmark} inevitably results in a quasi-Markovian process. That is, the obtained process is Markovian in some (generally obscure) state space. Here, we focus on a particular class of quasi-Markovian instances of  \eqref{eq:GLE:nonmark}, whose Markovian form can be related directly to a certain class of It\^o diffusion processes in an extended state space (see \cref{sec:extend:var:1}). Instead of attempting to directly discretize \eqref{eq:GLE:nonmark}, we numerically integrate the corresponding equivalent stochastic differential equation (SDE). This approach allows us in applications where the solution process of \eqref{eq:GLE:nonmark} is not quasi-Markovian (this is for example the case if $\K(t)$ has the form of a power law) to clearly separate the error due to time-discretization of the process --the focus of this article-- from the error induced by a quasi-Markovian approximation. 

The focus of this article is on the construction and analysis of integration schemes for the quasi-Markovian approximation. The proposed time discretizations (Sec. \cref{sec:numint}) of the equivalent SDE are constructed as  symmetric stochastic splitting schemes. The decomposition (of the associated generator) which is used as the basis for these stochastic splitting schemes is informed by results on stochastic splitting schemes for the underdamped Langevin equation \cite{bou2010long,Leimkuhler2013a,LeMaSt2015,abdulle2014high,Abdulle2015long}. 

In terms of the analysis we focus on the properties of the Markov chain obtained by such time-discretization of the quasi-Markovian approximation. That is, we discuss the existence of a stepsize dependent invariant measure $\muinv_{\deltat}$ of that Markov chain (here, $\deltat>0$ denotes the stepsize of the time discretization) and provide conditions for geometric convergence and the validity of a central limit theorem (\cref{sec:error:analysis}). Moreover, we provide a detailed analysis of the marginal measure $\muinv_{\deltat}(\dd \q)$ of the position variable $\q$ and its convergence to  the exact marginal measure $\muinv(\dd \q)$ as $\deltat \rightarrow 0$ (\cref{sec:num:error}). While the approximation accuracy of the measure $\muinv_{\deltat}(\dd \q)$ obviously is highly relevant for sampling applications, we emphasize that an accurate approximation of the marginal measure $\muinv(\dd \q)$ is also relevant for  accurately recovering dynamical properties  \cite{leimkuhler2016efficient,LeMaSt2015}.

Another aspect covered in this article is the behavior of the introduced stochastic splitting schemes in certain limits of parametrization (\cref{sec:limit:methods}). We show that the obtained numerical schemes behave consistently with well known homogenization results (summarized in \cref{sec:limit:dyn}) of the continuous dynamics and reduce to numerical integrators of the corresponding limiting dynamics with well known numerically favorable properties.  This has important implications for both sampling applications and for applications where the GLE is used as a dynamical model. In the former case our results ensure that the convergence order of the stepsize dependent error incurred in the invariant measure $\muinv_{\deltat}$ is not reduced in the respective limits. In fact we show that in the overdamped limit of \eqref{eq:GLE:nonmark} we obtain an increase of convergence order (fourth order instead of second order). A property which--in accordance with previous results \cite{Leimkuhler2013a,LeMaSt2015}--we refer to as {\em super-convergence}. In the latter case our results ensure that the dynamical properties of the simulated dynamics remain consistent with the underdamped/white-noise limit of \eqref{eq:GLE:nonmark}. 

In the remainder of this section we set up the basic framework for studying the Markovian reformulation of the GLE. We briefly review the above-mentioned homogenization results for the continuous dynamics as well as some results on the ergodic properties of the continuous dynamics. 

\subsection{Quasi-Markovian generalized Langevin equations (QGLE)}\label{sec:extend:var:1}
Consider the SDE defined on the extended space $\xDomain = \qDomain \times \pDomain \times \RR^{m}$,
\begin{equation}\label{eq:qgle}\tag{QGLE}
\begin{aligned}
\dot{\q} &= \M^{-1} \p ~ \dd t, \\
\begin{pmatrix}
\dot{\p}\\
\dot{\s}
\end{pmatrix} & =
\begin{pmatrix}
-\nabla_{\q}U(\q)\\
\0
\end{pmatrix}
-\Gammabf
\begin{pmatrix}
\M^{-1}\p\\
\s
\end{pmatrix}
+\sqrt{\beta^{-1}}\Sigmabf \dot{\W}, 
\end{aligned}
\end{equation}
where $\Gammabf,\Sigmabf$ are block matrices of the form
\begin{equation*}
\Gammabf := 
\begin{bmatrix}
   \Gammabf_{1,1} & \Gammabf_{1,2}\\
   \Gammabf_{2,1} & \Gammabf_{2,2}
\end{bmatrix}  \in \RR^{(n+m) \times (n+m)}, 
\hspace{0.1in} 
\Sigmabf := 
\begin{bmatrix}
\Sigmabf_{1}\\
\Sigmabf_{2}
\end{bmatrix}
=
\begin{bmatrix}
\Sigmabf_{1,1} & \Sigmabf_{1,2}\\
\Sigmabf_{2,1} & \Sigmabf_{2,2}
\end{bmatrix} \in \RR^{(n+m) \times (n+m)},
\end{equation*}
with $m\geq n$, and where  $\dot{\W}$ is a Gaussian white noise in $\RR^{n+m}$ with independent components, i.e., $\dot{\W}=[\dot{W}_{i}]_{1\leq i \leq n+m}$, such that, $\dot{W}_{i}\sim \mathcal{N}(0,1)$ and $\mathbb{E}[\dot{W}_{i}(t) \dot{W}_{j}(s)] = \delta_{ij} \delta(t-s)$.\\

In what follows we first provide a set of sufficient conditions which ensure that the SDE \eqref{eq:qgle} can be rewritten in the form of the stochastic integro-differential equation \eqref{eq:GLE:nonmark} and possesses an invariant measure which is such that its marginal in $\q,\p$ coincides with the Gibbs-Boltzmann distribution.  

\begin{assumption}\label{as:gle}
\hfill
\begin{enumerate}[label=(\roman*)]
\item \label{as:gle:it:1}
There exists a symmetric positive definite matrix $\Q \in\RR^{m\times m}$ such that,
\begin{equation}\label{eq:Lyapunov:Gamma:q}
\Gammabf 
\begin{pmatrix}
\I_{n} & \0 \\
\0 & \Q
\end{pmatrix}
 + 
 \begin{pmatrix}
\I_{n} & \0 \\
\0 & \Q
\end{pmatrix}
\Gammabf^{\trans} = \Sigmabf  \Sigmabf^{\trans}.
\end{equation}
\item \label{as:gle:it:2}
 the real parts of all eigenvalues of the matrix
\begin{equation}
\Gammabf_{\M}:=\Gammabf \, \begin{pmatrix}
\M^{-1} & \0 \\
\0 & \I_{m}
\end{pmatrix}
\end{equation}
are positive. That is $-\Gammabf_{\M} $ is a stable matrix.
\item the matrices $\Gammabf_{1,1}$ and $\M$ commute. 
\end{enumerate}
\end{assumption}
A derivation of the following proposition can be found in \cite{Ceriotti2010} (see also \cite{leimkuhler2017ergodic}).
\begin{proposition}\label{prop:GLE:cons}
Let  \cref{as:gle} hold and let $\Q \in \RR^{m\times m}$ be as specified therein. The SDE \cref{eq:qgle} 
conserves the probability measure $\muinv(\dd \q\,\dd \p \, \dd \s)$ with density
\begin{equation}\label{eq:invariant:measure}
\rho_{\Q,\beta}(\q,\p,\s) \propto e^{-\beta [ U(\q) + \frac{1}{2}\p^{\trans} \M^{-1}\p + \frac{1}{2}\s^{\trans} \Q^{-1} \s]}.\end{equation}
If further, $\s(0)\sim \mathcal{N}(\0,\Q)$ (independent of $\q(0),\p(0),\dot{\W}$), 
then the SDE  \cref{eq:qgle}  can be rewritten in the form of a stochastic integro-differential equation  \cref{eq:GLE:nonmark} with 
\begin{equation}\label{eq:auto-cov}
\K(t) = \Kbf_{\Gammabf}(t) := \Gammabf_{1,1} \delta(t) - \Gammabf_{1,2} e^{-t \Gammabf_{2,2}} \Gammabf_{2,1}.
\end{equation}
\end{proposition}

We refer to generalized Langevin equations whose memory kernel is of the form specified in \cref{eq:invariant:measure} as {\em Quasi-Markovian Generalized Langevin equations} (QGLEs). While more general parametrization of \eqref{eq:qgle} are possible, we focus in this article on two classes of memory kernels 
which can be characterized by some additional constraints on the form of $\Gammabf$ as summarized in \cref{as:gle:kernel}. Memory kernels falling into either of these two classes (or which are positive linear combinations of instances of either class) make up almost all parameterizations of \eqref{eq:qgle} appearing in the literature. 
\begin{assumption}\label{as:gle:kernel}
The matrices $\Gammabf$ and $\Q$ are such that either 
\begin{enumerate}[label=(\roman*)]
\item \label{as:gle:kernel:it:2}
$\Gammabf_{1,1} = \0$ and  $\Gammabf_{1,2 }\Q=  - \Gammabf_{2,1}^{\trans}$
\setcounter{mycounter}{\value{enumi}}
\end{enumerate}
or
\begin{enumerate}[label=(\roman*)]
\setcounter{enumi}{\value{mycounter}}
\item  \label{as:gle:kernel:it:1}
$\Gammabf_{1,1}$ is symmetric positive definite, $\Q= \I_{m}$, and $\Gammabf_{1,2}= \Gammabf_{2,1}^{\trans}$
\end{enumerate}
\end{assumption}

Parametrization according to \cref{as:gle:kernel}\,\ref{as:gle:kernel:it:2} allows for the representation of memory kernels of the form
\[
\K(t) =\I_{n} \sum_{l=1}^{R} c_{l} e^{-a_{l} t} \cos( b_{l} t ), \qquad c_{l}, a_{l} >0, \; b_{l} \geq 0, \; R \in \NN.
\] 
By choosing the matrices $\Gammabf_{1,2}$ and $\Q$ appropriately, one can additionally include cross-correlations between components. Typically such parameterizations are used in applications where the GLE is being treated as a dynamical model; see \cref{sec:int:dyn:model}.

Parametrization according to \cref{as:gle:kernel}\,\ref{as:gle:kernel:it:1} allows for the representation of memory kernels where $\K$ acts as a high-pass filter. Such parameterizations are most commonly used in sampling applications; see \cref{sec:int:sampling}. For example, the memory kernel $\K(t) = \I_{n} \left ( \gamma \delta(t) - \lambda \frac{\gamma}{\tau} e ^{-t/\tau} \right ), \lambda \in [0,1),\gamma>0, \tau>0$ used in \cite{Wu2015} can be represented in the form \eqref{eq:auto-cov} using $\Gammabf_{1,1} = \gamma  \I_{n}$, $\Gammabf_{1,2}=\Gammabf_{2,1} = \sqrt{{\lambda \gamma}/{\tau}} \I_{n} $ and $\Gammabf_{2,2} =  \frac{1}{\lambda}  \I_{n}$.

\subsubsection{Ergodicity and central limit theorem for QGLEs}\label{sec:ergodicity}
Under additional conditions on the coefficients of \eqref{eq:qgle}, the solution process is ergodic with unique invariant measure $\muinv$ and satisfies a central limit theorem. This follows from the fact that the semi-group of the associated evolution operators decays exponentially in a weighted $L_{\infty}$-norm.

In order to state this result, we need to first set some notation. 
For prescribed $\Kc \in \mathcal{C}^{\infty}(\xDomain, [1,\infty))$ with well-defined limit $\lim_{\left \|x \right\|\rightarrow \infty}\mathcal{K}(x) = \infty$, we define the set of functions
\[
L^{\infty}_{\Kc}(\xDomain) := \left \{ \varphi : \xDomain \rightarrow \RR, \text{ measurable } :\sup_{x \in \xDomain} \abs*{\frac{\varphi(x)}{\Kc(x)}} < \infty \right  \},
\]
which, when equipped with the norm $\varphi \mapsto \norm{\varphi}_{L^{\infty}_{\Kc}} :=\sup_{\x \in \xDomain} \abs*{\varphi(\x) / \Kc(\x)}$,
forms a Banach space. We denote by $\mathcal{C}_{P}(\xDomain,\RR) = \bigcap_{l \in \NN} L^{\infty}_{\Kc_{l}}(\xDomain)$ with $\Kc
_{l}(\x) = 1+\abs{\x}^{2l}$ the set of at most polynomially growing real-valued functions and by $\mathcal{C}_{P}^{p}(\xDomain,\RR)\subseteq \mathcal{C}_{P}(\xDomain,\RR)$ the set of real-valued function whose partial derivatives up to order $p\in \NN$ exist and grow at most polynomially, i.e., $\varphi \in \mathcal{C}_{P}^{p}(\xDomain,\RR) \iff \partial_{\x_{i_{1}}}\dots  \partial_{\x_{i_{m}}} \varphi \in \mathcal{C}_{P}(\xDomain,\RR)$ for any differential operator $\partial_{\x_{i_{1}}}\dots  \partial_{\x_{i_{m}}}$ with $m\leq p$. Unless stated otherwise, we consider operators introduced in the following to be defined on the core 
\begin{equation}\label{eq:def:testfunc}
\testfunc := \bigcap_{p=1}^{\infty}\mathcal{C}_{P}^{p}(\xDomain,\RR).
\end{equation}
In particular, the infinitesimal generator, $\Lcg$, of \cref{eq:qgle} when constrained to this set of test functions takes the form
\begin{equation}\label{def:gle:generator}
\Lcg = -\nabla_{\q}U(\q)\cdot \nabla_{\p} + \M^{-1}\p \cdot \nabla_{\q}- \Gammabf_{\M}\begin{pmatrix} \p \\ \s\end{pmatrix} \cdot \nabla_{\z} + \frac{\beta^{-1}}{2} \Sigmabf \Sigmabf^{\trans}  : \nabla^{2}_{\z},
\end{equation}
where $ \Sigmabf\Sigmabf^{\trans}  : \nabla^{2}_{\z} = \sum_{i=1}^{M} \sum_{j=1}^{M}\left [  \Sigmabf\Sigmabf^{\trans}  \right ]_{i,j} \partial_{\z_{i}}\partial_{\z_{j}}, ~M=n+m$. Here, as well as in the remainder of this article, $\z=(\p,\s)$ is used as  shorthand for the combined vector of momenta and auxiliary variables. 
We denote the formal $L^{2}$ adjoint (also known as the Fokker-Planck operator) of $\Lcg$ as $\Lcg^{\dagger}$. For $t\geq 0$, we denote the evolution operator associated with the SDE \cref{eq:qgle} as $e^{t\Lcg}$, i.e.,
$ \left ( e^{t\Lcg}\varphi\right )(x) = \EE \left [ \varphi(\x(t)) \given \x(0) = x \right ]$,
where  the expectation is with respect to the driving Wiener process, $\W$ of \cref{eq:qgle} and $\x=(\q,\p,\s)$ is used as a shorthand for the combined vector of positions, momenta, and auxiliary variables. 

\begin{assumption}\label{as:gle:2}
\hfill
\begin{enumerate}[label=(\roman*)]
\item \label{as:gle:it:3} The matrices $\Gammabf_{\M}$ and $\Sigmabf$ are such that the operator $\partial_{t}-\Lc_{\rm GLE}$ is hypoelliptic. (see \cite[Proposition 7]{leimkuhler2017ergodic} for sufficient algebraic conditions on $\Gammabf_{\M},\Sigmabf$ for Hypoellipticity of the operator.) In particular,
\[
\RR^{M} = \bigcup_{k=0}^{M}  \bigcup_{i=0}^{M}  \Gammabf_{\M}^{k} \Sigma_{i}, \quad \Sigmabf = ( \Sigma_{1},\dots,\Sigma_{M}),\quad M=n+m.
\]
\item \label{as:gle:it:4} The potential function $U$ is of the form $U(\q) = \Ug(\q) + \Ub(\q)$,
where $\Ug(\q) = \frac{1}{2} \q^{\trans} {\bm \Omega} \q$, with ${\bm \Omega}\in \RR^{n\times n}$ being a symmetric positive definite matrix, and $\Ub \in \mathbb{C}^{\infty}(\RR^{n},\RR)$ is such that its derivatives are uniformly bounded in $\RR^{n}$, i.e.,  $\sup_{\q \in \RR^{n}}\norm{\partial_{\q_{i_{1}}}\partial_{\q_{i_{2}}}\cdots\partial_{\q_{i_{k}}} \Ub(\q)} < \infty$ for any $k\in \NN$, and $i_{1},\dots,i_{k} \in \{1,\dots,n\}$. 
\end{enumerate}
\end{assumption}

Under the above stated assumptions exponential convergence of the associated semi-group and a central limit theorem for trajectory averages can be established:
\begin{proposition}[\cite{leimkuhler2017ergodic}]\label{prop:exp:conv}
Let \cref{as:gle,as:gle:kernel,as:gle:2} be satisfied. 
\begin{enumerate}[label=(\roman*)]
\item 
For any $l\in \mathbb{N}$, there exist constants $\kappa_{l}>0, C_{l}>0$ such that
\begin{equation*}\label{eq:L2:exp:conv}
\hspace{-.1in}
\forall t \geq 0,~ \forall \varphi \in L^{\infty}_{\Kc_{l}}, ~\norm*{{\rm e}^{t\Lcg}\varphi - \int \varphi \,\dd \muinv }_{L^{\infty}_{\Kc_{l}}} \leq C_{l}{\rm e}^{-t\kappa_{l} } \norm*{ \varphi - \int \varphi \,\dd \muinv }_{L^{\infty}_{\Kc_{l}}}.
\end{equation*}
\item Let $\overline{\varphi}_{t} :=  t^{-1}\int_{0}^{t}\varphi(\x(t)) \dd t$. If $\varphi \in L_{\Kc_{l}}^{\infty}(\xDomain)$ for some $l\in \NN$, then there is a finite $\sigma_{\varphi}^{2}>0$ so that 
\begin{equation}\label{eq:clt}
\sqrt{t} \left (\overline{\varphi}_{t}-  \int \varphi \, \dd \muinv \right ) \xrightarrow[t \to +\infty]{\mathrm{law}}  \mathcal{N}(0,\sigma_{\varphi}^{2}).
\end{equation}
\end{enumerate}
\end{proposition}

If not explicitly stated otherwise, we assume throughout the remainder of this article that the parameterization of \eqref{eq:qgle} is such that \cref{as:gle,as:gle:kernel,as:gle:2} are all satisfied.

\subsection{Limiting dynamics}\label{sec:limit:dyn}

The underdamped Langevin and overdamped Langevin dynamics can be obtained as limiting dynamics of the Markovian reformulation of the GLE. In what follows we briefly review two key results from \cite{Ottobre2011} and \cite{lim2017homogenization}, which we will later show to hold in slightly modified form for the discretized dynamics. For this purpose we consider the following rescaled process, which is obtained from \cref{eq:qgle} by a change of variable corresponding to a time rescaling as $(\q^{\lambda}(t),\p^{\lambda}(t),\s^{\lambda}(t))= (\q(\lambda t),\p(\lambda t),\s(\lambda t))$, with $\lambda >0$.

\begin{equation}\label{eq:GLE:mark:rescaled}\tag{QGLE-scaled}
\begin{aligned}
\dot{\q}^{\lambda} &= \lambda \p^{\lambda}, \\
\begin{pmatrix}
\dot{\p}^{\lambda} \\
\dot{\s}^{\lambda}
\end{pmatrix}
& =
\left [ 
\lambda
\begin{pmatrix}
-\nabla_{\q}U(\q^{\lambda})\\
\0
\end{pmatrix}
-\lambda\Gammabf^{\mu} 
\begin{pmatrix}
\p^{\lambda}\\
\s^{\lambda}
\end{pmatrix} \right ]
+\sqrt{\lambda\beta^{-1}}\Sigmabf^{\mu} \dot{\W}, 
\end{aligned}
\end{equation}
with $\s^{\lambda}(0) \sim \mathcal{N}(\0,\I_{n})$ and 
\begin{equation}\label{eq:Gamma:rescale}
\Gammabf^{\mu} = 
 \begin{pmatrix}
\0 & -\mu_{1}{\bm D}_{a}\\
\mu_{1}{\bm D}_{a}& \mu_{2}{\bm D}_{b}
\end{pmatrix}\in \RR^{2n\times2n},
\qquad
\Sigmabf^{\mu} = 
 \begin{pmatrix}
\0 &  \0 \\
\0 & \sqrt{2 \mu_{2}}{\bm D}_{b}
\end{pmatrix}\in \RR^{2n\times2n},
\end{equation}
where ${\bm D}_{a} := \diag(a_{1}, a_{2}, \dots a_{n})$ and ${\bm D}_{b} := \diag(b_{1}, b_{2}, \dots b_{n})$ are positive diagonal matrices. In the view of the stochastic integro-differential equation \eqref{eq:GLE:nonmark} the form of the matrices $\Gammabf^{\mu},\Sigmabf^{\mu}$ corresponds to a rescaling of the memory kernel $\K(t) =  \diag( a_{1}^{2} {\rm e}^{-t b_{1}},\dots, a_{n}^{2} {\rm e}^{-t b_{n}} )$ as $\K^{\mu}(t) = \mu_{1}^{2}\K(\mu_{2} t)$. Without loss of generality we assume $\M=\I_{n}$ here.\footnote{ Note that \eqref{eq:qgle} can always be transformed to a system with a isotropic mass matrix by applying a change of variable of the form $\q \mapsto \M^{1/2}\q $, $\p \mapsto \M^{-1/2}$ to the process resulting in modifications  as  $-\nabla U \mapsto -\M^{-1/2} \nabla U$, $\Gammabf \mapsto \diag(\M^{-1/2},\I_{m})\,\Gammabf\, \diag(\M^{-1/2},\I_{m})$, $\Sigmabf\mapsto \diag(\M^{-1/2},\I_{m})\Sigmabf$, of the force, friction matrix and diffusion matrix, respectively. }

For the parameter choice $ \lambda = 1, \mu_{1} =\epsilon^{-1}, \mu_{2}=\epsilon^{-2}$, the rescaled process \eqref{eq:Gamma:rescale} converges weakly to the solution of an underdamped Langevin equation as $\epsilon\rightarrow 0$ as detailed in the following proposition. 

\begin{proposition}[White noise limit, \cite{Ottobre2011}]\label{prop:wn:limit}
Let $ \lambda = 1, \mu_{1} =\epsilon^{-1}, \mu_{2}=\epsilon^{-2}$. 
 For finite $T>0$ we have 
\[
(\q^{\lambda}(t),\p^{\lambda}(t))  \xrightarrow[\varepsilon \to 0]{\mathrm{law}}  (\q(t),\p(t)),
\]
uniformly in $t\in [0,T]$, where $(\q(t),\p(t))_{t\geq 0}$ denotes the solution process of the underdamped Langevin equation \eqref{eq:LD} with $\GammaLD = {\bm D}_{a}^{2}{\bm D}_{b}^{-1}$, $\M=\I_{n}$, and initial values $(\q(0),\p(0))= (\q^{\lambda}(0),\p^{\lambda}(0))$.
\end{proposition}

Similarly, when considering the scaling  $\lambda = \mu_{1} =\mu_{2}=\epsilon^{-1}$ one can show that the solution of \eqref{eq:GLE:mark:rescaled} converges weakly to the solution of an overdamped Langevin equation of the form
\begin{equation}\label{eq:OLD}\tag{BD}
\dot{\q} = -\Lambda \nabla U(\q) + \sqrt{2\beta^{-1}}\Lambda^{1/2} \dot{\W}.
\end{equation}

\begin{proposition}[Overdamped limit]\label{prop:od:limit}
Let $\lambda = \mu_{1} =\mu_{2}=\epsilon^{-1}$. $
\s^{\lambda}(0) \sim \mathcal{N}(\0,\I_{n})$, and $(\q^{\lambda}(0), \p^{\lambda}(0))\in \RR^{2n}$. For finite $T>0$ we have 
\[
\q^{\lambda}(t)  \xrightarrow[\varepsilon \to 0]{\mathrm{law}} \q(t),
\]
uniformly in $t\in [0,T]$, where $(\q(t))_{t\geq 0}$ denotes the solution process of the overdamped Langevin equation \eqref{eq:OLD} with $\Lambda = \left ({\bm D}_{a}^{2}{\bm D}_{b}^{-1}\right)^{-1}$.
\end{proposition}
\begin{proof}
This result is a direct consequence of Theorem IV.1 in \cite{lim2017homogenization}. The same limit has been also studied in a less general setup in \cite{schuss2010diffusion}.
\end{proof}

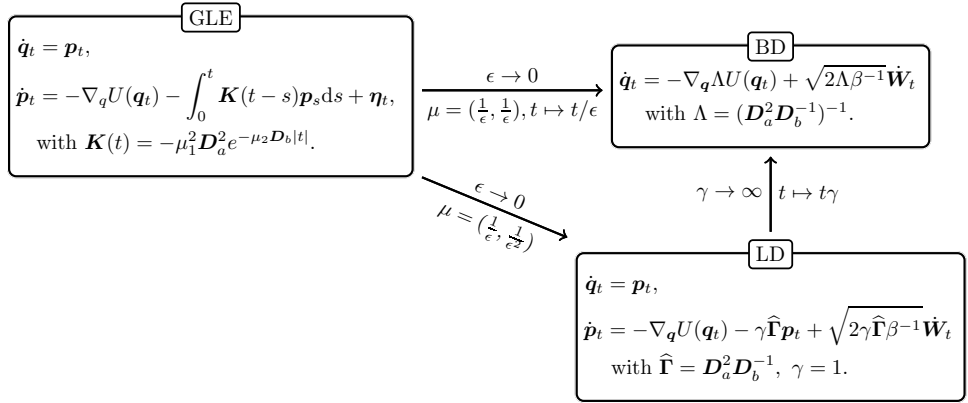
\begin{figure}
\begin{tikzpicture}[thick,scale=0.8, every node/.style={scale=0.8}]
    \node (GLE) [align=left] {\begin{eqtbox}[%
        GLE,
                width=.52\textwidth
    ]
    $\begin{aligned}
   \dot{\q}_{t} &= \p_{t},\\ 
    \dot{\p}_{t} &= -\nabla_{q} U(\q_{t}) - \int_{0}^{t}\K(t-s)\p_{s} {\rm d} s  + \etabf_{t},\\
    &\text{with }  \K(t) =  - \mu_{1}^{2}{\bm D}_{a}^{2} e^{- \mu_{2}{\bm D}_{b}|t|}.
    \end{aligned}$
    \end{eqtbox}};
     \node (BD) [right=2.4cm of GLE,align=left]{\begin{eqtbox}[%
         BD,
                width=.41\textwidth
    ]
     $   \begin{aligned}
   \dot{\q}_{t} &= -\nabla_{\q}\Lambda U(\q_{t})+ \sqrt{2 \Lambda\beta^{-1}} \dot{ \W}_{t}\\
   &\text{ with } \Lambda =  (  {\bm D}_{a}^{2}{\bm D}_{b}^{-1} )^{-1}.
      \end{aligned}$
      
    \end{eqtbox}};
    \node (LD)  [below=1cm of BD,align=right] {\begin{eqtbox}[%
        LD,
                width=.5\textwidth
    ]
     $  \begin{aligned}
   \dot{ \q}_{t} &=  \p_{t},\\
   \dot{\p}_{t} &= -\nabla_{\q} U(\q_{t}) - \gamma \GammaLD \p_{t} + \sqrt{2\gamma\GammaLD{\beta}^{-1}}\dot{\W}_{t}\\
   &\text{ with } \GammaLD = {\bm D}_{a}^{2}{\bm D}_{b}^{-1},~ \gamma =1.
      \end{aligned}$
    \end{eqtbox}};
   \draw[->] (GLE)--(BD) node [midway,above,sloped] {$\epsilon \rightarrow 0$};
   \draw[->] (GLE)--(BD) node [midway,below,sloped] {$\mu=(\frac{1}{\epsilon},\frac{1}{\epsilon}), t \mapsto t/\epsilon$};
\draw[->] (LD)--(BD) node [midway,right]  {$t \mapsto t \gamma $}; 
\draw[->] (LD)--(BD) node [midway,left] {$\gamma \rightarrow \infty$};
\draw[->] (GLE)--(LD) node [midway,above,sloped] {$\epsilon \rightarrow 0$};
\draw[->] (GLE)--(LD) node [midway,below,sloped]{$\mu=(\frac{1}{\epsilon},\frac{1}{\epsilon^{2}})$};  
\end{tikzpicture}
\caption{Diagram of Langevin limits. The overdamped limit  of the GLE (GLE $\rightarrow$ BD) is detailed in \cref{prop:od:limit}. The white noise limit (GLE $\rightarrow$ LD) is detailed in \cref{prop:wn:limit}. The overdamped limit of the underdamped Langevin equation (LD $\rightarrow$ BD) is a well known result in the literature (see, e.g.,\cite{Pavliotis2014}).  } 
\end{figure}

\section{Symmetric stochastic splitting methods for the QGLE}\label{sec:numint}
In this section we present the basic construction and implementation of the class of proposed stochastic splitting schemes for the Markovian reformulation  \cref{eq:qgle} of the GLE as well as an elementary analysis of the incurred weak error. 

\subsection{Construction of numerical methods based on splitting}
As mentioned in the introduction, we construct  splitting schemes using a similar procedure as that employed  for the underdamped Langevin equation in \cite{Leimkuhler2013a,LeMaSt2015}. 
Such schemes are based on a decomposition of the generator of the underdamped Langevin equation as 
$
\LcLD = \Lca + \Lcb+ \LcLDo,
$
where $\Lca= -\nabla U(\q)\cdot \nabla_{\p}$, $\Lcb= -\p \cdot \M^{-1}\nabla_{\q}$, and $\LcLDo = -\GammaLD  \M^{-1}\cdot \nabla \p +\beta^{-1}\GammaLD   : \nabla^{2}_{\p}$. By applying a Strang splitting with stepsize $\deltat>0$ twice (typically first to treat the Liouville operator $\Lca + \Lcb$ associated with the Hamiltonian vector field, and then subsequently to compute the combination of that operator with the term $\Lco$), a symmetric stochastic splitting scheme with associated evolution operator is obtained:
\[
\hat{\Pc}^{{\textup{ld-OBABO}}}_{\deltat} = \exp \left ( (\deltat/2) \LcLDo\right ) \exp\left ( (\deltat/2) \Lcb\right )\exp\left ( \deltat \Lca\right )\exp\left ( (\deltat/2)  \Lcb\right )\exp\left ( (\deltat/2)  \LcLDo\right ).
\] 
Similarly, by either changing the ordering within the Strang splitting or by changing the pair of operators selected for the first application of the Strang splitting, other splitting schemes can be obtained which are uniquely identified by palindromes of the form ${\rm XYZYX}$ where ${\rm X,Y,Z}\in \{ {\rm A}, {\rm B}, \widehat{\rm O} \} $ are distinct placeholders.   (The symmetry of this decomposition is not essential but typically improves the accuracy and efficiency of the resulting scheme with little added computational cost.)
 
This construction can be easily generalized to the Markovian reformulation of the GLE \cref{eq:qgle} by using the fact that \eqref{eq:qgle} structurally resembles the underdamped Langevin equation \eqref{eq:LD}. That is, we consider a decomposition of $\Lcg$ as $\Lcg = \Lca + \Lcb+ \Lco,
$
where $\Lca,\Lcb$ are defined as above and 
\[
\Lco = - \Gammabf\begin{pmatrix} \M^{-1}\p \\ \s\end{pmatrix} \cdot \nabla_{\z} + \frac{\beta^{-1}}{2} \Sigmabf \Sigmabf^{\trans}  : \nabla^{2}_{\z}.
\]
The only difference between this decomposition and the decomposition of the operator $\LcLD$ is that $\Lco$ corresponds to the generator of a linear SDE in $\p$ and $\s$, whereas  the operator $\LcLDo$ in the otherwise identical decomposition of $\LcLD$ is the generator of a linear SDE in $\p$ only. Thus, symmetric splitting schemes for \eqref{eq:qgle} can be constructed in the same way as for the underdamped Langevin equation resulting in numerical integrators with associated evolution operators of the form 
\begin{equation}\label{eq:propagator}
\hat{\Pc}^{{\textrm{gle-XYZYX}}}_{\deltat} = \exp \left ( \frac{\deltat}{2} \Lc_{{\rm X}}\right ) \exp\left ( \frac{\deltat}{2} \Lc_{{\rm Y}}\right )\exp\left ( \deltat \Lc_{{\rm Z}}\right )\exp\left ( \frac{\deltat}{2}  \Lc_{{\rm Y}}\right )\exp\left ( \frac{\deltat}{2}  \Lc_{{\rm X}}\right ),
\end{equation}
where ${\rm X,Y,Z}\in \{ {\rm A}, {\rm B}, {\rm O} \} $ are again distinct placeholders.

\subsection{Implementation}\label{sub:splitting:implentation}
By construction the numerical integrator for the associated evolution operator $\hat{\Pc}^{{\textrm{gle-XYZYX}}}_{\deltat}$ is of the form 
\begin{equation}
\hat{\Phi}^{{\textrm{gle-XYZYX}}}_{\deltat} = \Phi^{\rm X}_{\deltat/2} \circ \Phi^{\rm Y}_{\deltat/2} \circ \Phi^{\rm Z}_{\deltat} \circ \Phi^{\rm Y}_{\deltat/2} \circ \Phi^{\rm X}_{\deltat/2},
\end{equation}
where $\Phi^{\rm X}_{\deltat}, {\rm X} \in \{{\rm A},{\rm B},{\rm O}\}$ are the solutions maps of the differential equations associated with the operators $\Lc_{{\rm X}}, {\rm X} \in \{{\rm A},{\rm B},{\rm O}\}$, respectively. A practical implementation of the above-described splitting schemes therefore requires that each differential equation associated with the operators $\Lca,\Lcb$ and $\Lco$ can be solved exactly. Indeed, in the case of the operators $\Lca$ and $\Lcb$ the solution of the associated differential equations $\dot{\q}= \M^{-1}\p$, and $\dot{\p}= -\nabla U(\q)$ correspond to Euler updates of the form 
\begin{equation}
\Phi^{\rm A}_{\deltat}: (\q,\p,\s)  \mapsto  (\q + \deltat \M^{-1} \p, \p, \s), ~~\Phi^{\rm B}_{\deltat}: (\q,\p,\s)  \mapsto  (\q, \p - \deltat\nabla_{\q} U(\q), \s),
\end{equation}
respectively. The solution of the SDE associated with the operator $\Lco$,
\begin{equation}
\dot{\z}
 =
-\Gammabf_{\M}\z
+\sqrt{\beta^{-1}}\Sigmabf \dot{\W}, \\
\end{equation}
coincides in law with $\z(\deltat) =   \F\z(0) + \S \Rand$, (see e.g. \cite{Gardiner,Pavliotis2014}) where $\Rand \sim \mathcal{N}( {\bm 0}, \I_{n+m})$ denotes a vector of independent and standard normal distributed random variables in $\RR^{n+m}$, $\F = \exp(-\deltat\Gammabf_{\M})$ is the matrix exponential of the matrix $-\deltat\Gammabf_{\M}$, and $\S$ solves the equation 
\[
\S\S^{\trans} = \beta^{-1} \left [ \begin{pmatrix} \M & \0 \\ \0 & \Q \end{pmatrix} - \F\begin{pmatrix} \M & \0 \\ \0 & \Q \end{pmatrix} \F^{\trans} \right ].  
\]
The corresponding stochastic flow map which updates the combined state vector $\x$ accordingly, is of the form 
\[
\Phi^{{\rm O}}_{\deltat}:  (\q,\z) \mapsto  \left (\q, \F\; \z + \S \Rand \right ),
\]
where $\Rand$ is independently resampled at every application of $\Phi^{{\rm O}}_{\deltat}$. With the definition of the updates  $\Phi^{{\rm A}}_{\deltat}, \Phi^{{\rm B}}_{\deltat}, \Phi^{{\rm O}}_{\deltat}$ at hand, one can find explicit algorithmic forms for the integration map $\hat{\Phi}^{{\textrm{gle-XYZYX}}}_{\deltat}$. We provide an algorithmic implementation of $\hat{\Phi}^{{\textrm{gle-BAOAB}}}_{\deltat}$ in  {\em Algorithm} \ref{alg:gle:baoab}.

\begin{algorithm2e}
\caption{ gle-BAOAB}
\label{alg:gle:baoab}
\SetKwData{Left}{left}\SetKwData{This}{this}\SetKwData{Up}{up}
\SetKwFunction{Union}{Union}\SetKwFunction{FindCompress}{FindCompress}
\SetKwInOut{Input}{input}
\SetKwInOut{Output}{output}
\Input{$(\q,\p,\s)$}
$\p \gets  \p -  \frac{\deltat}{2} \nabla U(\q)$\;
$\q \gets \q +  \frac{\deltat}{2} \M^{-1}\p$\;
$\begin{bmatrix}\p \\ \s \end{bmatrix} \gets \F \begin{bmatrix}\p \\ \s \end{bmatrix} + \S \Rand$\;
$\q \gets \q + \frac{\deltat}{2} \M^{-1}\p$\;
$\p \gets  \p - \frac{\deltat}{2}\nabla U(\q)$\;
\Output{$(\q,\p,\s)$}
\BlankLine
\end{algorithm2e}

\subsection{Order of convergence of weak error}\label{sec:weak:error}
A numerical scheme with associated evolution operator $\Pc_{\deltat}$ is said to have global weak order $p$ when applied to \eqref{eq:qgle} with $\x(0)\sim \muinv_{0}$ if, for all $\varphi \in \mathcal{C}^{\infty}_{P}(\xDomain,\RR)$, and for all $T>0$, there exists a constant $C(T,\varphi)>0$, such that
\[
\abs*{ \left ( \Pc_{\deltat}^{n}\varphi \right )(x) - ({\rm e}^{ n \deltat \Lcg } \varphi ) (x) } \leq \deltat^{p}C(T,\varphi),
\]
for all $t_{n}=n\deltat \in [0,T]$, for $\muinv_{0}$-almost all $x$, and for all sufficiently small $\deltat>0$. Here, as well as in sequel, $\Pc^{n} = \Pc^{n-1} \Pc$ with $\Pc^{0}={\rm Id}$, denotes the $n$-th power of the evolution operator $\Pc$. The above 
discretization schemes all have weak order $p=2$:
\begin{proposition}
Let \cref{as:gle:2} be satisfied. Then, any symmetric stochastic splitting schemes with evolution operator $\hat{\Pc}^{{\textnormal{gle-XYZYX}}}_{\deltat}$ has global weak order 2.
\end{proposition}
\begin{proof}
This result is a direct consequence of Theorem 2 in \cite{mil1986weak} which provides a set of sufficient conditions for the local weak error to coincide with that of the global weak order. We therefore only provide a brief outline of the proof. By Taylor expanding (at $\deltat=0$) both $({\rm e}^{ n \deltat \Lcg } \varphi ) (x)$  and $\left ( \Pc_{\deltat}^{n}\varphi \right )(x)$ with $\Pc_{\deltat}=\hat{\Pc}^{{\textrm{gle-XYZYX}}}_{\deltat}$, and comparing powers in $\deltat$   we obtain an expansion of the local weak error as
\[
\begin{aligned}
\abs*{ \left ( \Pc_{\deltat} \varphi \right )(x) - ({\rm e}^{\deltat \Lcg } \varphi ) (x) } &= h^{3}\Acal_{3} \varphi (x) +  \deltat^{5}r_{\varphi,\delta,5}(x), 
\end{aligned}
\]
with 
\begin{equation}\label{eq:ac2}
\begin{aligned}
\Acal_{3} &= \frac{1}{12} \Big ( \left [ \Lcz, \left [ \Lcz,\Lcy \right ] \right] + \left [ \Lcy + \Lcz, \left [ \Lcy + \Lcz, \Lcx \right ] \right ]  \\
&  \qquad- \frac{1}{2} \left [ \Lcy, \left [ \Lcy,\Lcz \right] \right]- \frac{1}{2} \left [ \Lcx, \left [ \Lcx,\Lcy+\Lcz \right ] \right ]  \Big ),
\end{aligned}
\end{equation}
where $[A,B] = AB - BA$ denotes the commutator of the two linear operators $A,B$. In other words the convergence order of the local weak error is 2. \cref{as:gle:2} ensures that for $\varphi \in \testfunc$ the remainder term $r_{\varphi,\deltat,5}$ as well as $\Acal_{3} \varphi$ are both contained in $\testfunc$. For sufficiently small stepsize $\deltat>0$ the existence of a suitable Lyapunov functions (see proof of \cref{prop:ergodicity:discrete}) ensures that moments, $\EE \big [\norm{\xnum_{k}}^{2m} \big ]$, of any order $m\in \NN$ are uniformly bounded in the iteration index $k\in \NN$. All together, the conditions of \cite[Theorem 2]{mil1986weak} are met, which implies that the global weak convergence order coincides with the local weak convergence order. 
\end{proof}

\section{Error analysis of ergodic averages}\label{sec:error:analysis}
As in the case of the underdamped Langevin equation and the overdamped Langevin equation, Markov chains of the discretized dynamics
\begin{equation}\label{eq:iterate}
\xnum_{k+1} = \hat{\Phi}^{{\textrm{gle-XYZYX}}}_{\deltat} (\xnum_{k}), ~\xnum_{0} = \x(0), ~k \in \NN,
\end{equation}
can be used as  approximate Markov chain Monte Carlo methods for the computation of expectations with respect to the extended Gibbs-Boltzmann distribution $\muinv$. That is, expectations of observables $\varphi \in L^{2}(\muinv)$, where $L^{2}(\muinv) := \{ \varphi : \xDomain \rightarrow \RR \; \text{measurable} \mid \int \varphi^{2}\dd \muinv\}$, are approximately computed as  trajectory averages of the form $\widehat{\varphi}_{N}:=\frac{1}{N}\sum_{k=1}^{N-1}\varphi(\xnum_{k})$ from a finite trajectory $(\x_{k})_{k=1,\dots,N}$. Such approximate computations are performed under the premise that $(\xnum_{k})_{k\in \NN}$ is ergodic with respect to the  invariant measure $\muinv_{\deltat} \approx \muinv$, so that 
\[
\lim_{N \rightarrow \infty} \widehat{\varphi}_{N} = \int   \varphi(\x)  \muinv_{\deltat}(\dd \x),
\]
for almost all realizations of the Markov chain $(\xnum_{k})_{k \in \NN}$.
In this section we provide theoretical justification for such a computation by showing that the above-mentioned assumptions are indeed satisfied. We first show in  \cref{sec:num:ergodic} that the proposed numerical schemes result in ergodic Markov chains and the validity a central limit theorem for the Monte Carlo error in the following decomposition of the approximation error
\begin{equation}
\widehat{\varphi}_{N}-\EE_{x\sim \muinv}\left [ \varphi(x) \right ]  =  \underbrace{\left (\widehat{\varphi}_{N} - \EE_{x\sim \muinv_{\deltat}}\left [ \varphi(x) \right ]  \right )}_{\text{Monte Carlo error}} + \underbrace{\left (\EE_{x\sim \muinv_{\deltat}}\left [ \varphi(x) \right ] -  \EE_{x\sim \muinv}\left [ \varphi(x) \right ]\right )}_{\text{Systematic bias}}.
\end{equation}
In the next subsection (\cref{sec:num:error}) we provide an analysis of the stepsize dependent systematic bias.

\subsection{Ergodic properties and central limit theorem}\label{sec:num:ergodic}
In addition to showing the existence and uniqueness of the invariant measure $\muinv_{\deltat}$, we show geometric ergodicity of the Markov chain $(\xnum_{k})_{k \in \NN}$. Geometric ergodicity is equivalent to exponential convergence of the corresponding evolution operator  $\hat{\Pc}^{{\textrm{gle-XYZYX}}}_{\deltat}$ in some suitably weighted $L_{\infty}$ space. By \cite{Bhattacharya1982}, the latter property implies the validity of a central limit theorem.
\begin{theorem}\label{prop:ergodicity:discrete}
Let \cref{as:gle} and \cref{as:gle:2} be satisfied and let $\Pc_{h} = \hat{\Pc}^{{\textnormal{gle-XYZYX}}}_{\deltat}$.  Fix $l\in \NN, l >0$ and consider 
\begin{equation}\label{eq:def:lya}
\mathcal{K}_{l}(\q,\p,\s) = \left ( \x^{\trans} \C \x  \right )^{l} + 1, ~ l \in \mathbb{N},
\end{equation}
where $\C \in \RR^{n+m}$ is a suitably chosen symmetric positive definite matrix (see \cref{subsec:lya} for details) Then, there exists $\deltat^{*}>0$, such that for any $\deltat \in (0,\deltat^{*})$
\begin{enumerate}
\item  the Markov chain associated 
with $\Pc_{\deltat}$ has a unique invariant probability measure $\muinv_{\deltat}$, which admits a density with respect to the Lebesgue measure on $\xDomain$ and has finite moments, i.e.,
\begin{equation}
\int_{\xDomain}\Kc_{l}\dd \muinv_{\deltat} < \infty, \quad \text{ for all }  \; \;l \in \NN.
\end{equation}
\item 
there are constants $C_{l} >0$, $r_{l} \in (0,1)$ such that
\begin{equation}
\forall \,\varphi \in L_{\Kc_{l}}^{\infty} \quad \forall \,k\in \NN, \quad  \norm*{ (\Pc_{\deltat}^{k} \varphi) -  \EE_{\muinv_{\deltat}} \varphi }_{L^{\infty}_{\Kc_{l}}} \leq C_{l} r_{l}^{k}\norm{\varphi}_{L^{\infty}_{\Kc_{l}}}.
\end{equation}
 \end{enumerate}
 \end{theorem}

A complete proof of this result can be found in \cref{sec:proofs}. Here, we provide a brief outline:
the proof of the theorem relies on an application of Theorem 1.2 of \cite{Hairer2011a} (see also \cite{Meyn1997,Rey-Bellet2006a} for similar results) and as such includes the standard steps commonly followed for proving geometric ergodicity of a Markov chain. We first show that under the conditions of \cref{prop:ergodicity:discrete} a uniform minorization condition is satisfied. That is,
\begin{assumption}[Minorization condition]\label{as:umc}
Fix any $x_{\rm max} > 0$. There exist $\deltat^{*}>0$ such that for any $\deltat\in (0,\deltat^{*})$, there is $\alpha>0$ so that
\begin{equation}\label{eq:as:umc}
\forall \,\varphi \in \mathcal{C}_{0}(\xDomain,\RR) \qquad \inf_{|\x| \leq x_{\rm max}}\left(\Pc^{k}_{\deltat}\varphi \right ) (\x) \geq \alpha\int_{\xDomain}\varphi(x) \nu(d x),
\end{equation}
where $k=2$, and $\nu$ denotes the Lebesgue measure on $\xDomain$.
\end{assumption}
The validity of this minorization condition ensures that within any compact ball which is centered at the origin, the Markov chain $(\x_{k})_{k\in \NN}$ is mixing within a finite number of steps. As such it already ensures irreducibility of the Markov chain and thus guarantees the uniqueness of the invariant measure $\muinv_{\deltat}$ provided the latter exists. In order to ensure the existence of an invariant measure and exponential converge to that measure, the existence of a suitable Lyapunov function is shown in the second step of the proof. That is,
\begin{assumption}[Uniform Lyapunov condition]\label{as:ulc}
For any $l \in \mathbb{N},\;l>0$, there exists $\deltat^{*}>0$ and $a_{l},b_{l}>0$ such that for any $\deltat,~0 < \deltat \leq \deltat^{*}$, 
\[
\Pc_{\deltat}\mathcal{K}_{l} \leq e^{-a_{l} \deltat}\mathcal{K}_{l} +b_{l} \deltat.
\]
\end{assumption}
Given the validity of both \cref{as:umc} and \cref{as:ulc}, the remaining statements of \cref{prop:ergodicity:discrete} then follow as a consequence of \cite[Theorem 1.2]{Hairer2011a}.

By \cite[Corollary 2.26]{Lelievre2016a} exponential convergence in the sense of \cref{prop:ergodicity:discrete} implies that the operator $({\rm Id} - \Pc_{\deltat})$ when constrained to the subspace 
\[
L_{\Kc_{l},0}^{\infty}=\left \{ \varphi \in L_{\Kc_{l}}^{\infty} ~:  \int \varphi \muinv_{\deltat} = 0 \right \},
\] 
is invertible and the corresponding inverse operator is bounded in terms of the operator norm induced by $\norm{\cdot}_{L_{\Kc_{l}}^{\infty}}$. By the results in \cite{Bhattacharya1982}, this is sufficient for a functional central limit theorem to hold (\ref{eq:clt:2}).

\begin{corollary}[Central Limit theorem]
Let $l\in \NN $ and $\varphi \in  L_{\Kc_{l}}^{\infty} $. For sufficiently small $\deltat>0$, there is finite $\widehat{\sigma}_{\varphi}^{2}>0$ so that 
\begin{equation}\label{eq:clt:2}
\sqrt{N} \left (\widehat{\varphi}_{N}-  \int \varphi \, \dd \muinv_{\deltat} \right ) \xrightarrow[N \to +\infty]{\mathrm{law}}  \mathcal{N}(0,\widehat{\sigma}_{\varphi}^{2}).
\end{equation}
\end{corollary}

\subsection{Analysis of the systematic bias}\label{sec:num:error}
In this section we provide results regarding the convergence order in $\deltat$ of the discretization bias in ergodic averages of  symmetric splitting schemes. As discussed in \cref{sec:weak:error}, the weak convergence order of these schemes is two, which together with the above shown ergodicity result implies that also the convergence order of the systematic discretization bias in ergodic averages is at least two, i.e.,
\[
\lim_{N \rightarrow \infty} \widehat{\varphi}_{N} = \int   \varphi(\x)  \muinv_{\deltat}(\dd \x) = \int \varphi(\x) \muinv(\dd \x) + O(\deltat^{2}),
\]
as $\deltat \rightarrow 0$. In what follows we discuss two special cases where the second order convergence can be improved upon. We first derive the explicit form of the measure $\muinv_{\deltat}$ of the gle-BAOAB scheme in the situation where the target measure $\muinv$ is Gaussian. Secondly, we analyze the behavior of the discretization error of the gle-BAOAB scheme in the overdamped limit.

\subsubsection{Systematic bias for quadratic potentials}
An important property of the gle-BAOAB scheme is that its invariant measure $\muinv_{\deltat}$, when applied to a system with quadratic potential, is such that the marginal in $\q$, $\muinv_{\deltat}(\dd \q)$, coincides with the marginal in $\q$ of the exact invariant measure $\muinv(\dd \q) \propto e^{-\beta\q \Omegabf^{-1}\q}\,\dd\q $, so that
\[
\lim_{N\rightarrow \infty}\widehat{\varphi}_{N} =  \int   \varphi(\q)  \muinv_{\deltat}(\dd \q),
\]
for $\varphi \in L^{2} \left (\muinv(\dd \q) \right)$. More specifically, we have the following theorem. \begin{theorem}\label{prop:gaussian:error}
Let $U(\q) = \frac{1}{2}\q^{\trans}\Omegabf\;\q$ with $\Omegabf \in \RR^{n \times n}$ symmetric  positive definite. The Gaussian measure $\muinv_{\deltat}(\dd \x) 
\propto \exp\left({-(\x-\mubf_{\deltat})^{\trans}{\bm V}^{-1}_{\deltat}(\x-\mubf_{\deltat})}\right) \dd \x$, with
\begin{equation}\label{inv:var:BAOAB}
\mubf_{h} = {\bm 0},~~{\bm V}_{\deltat} = \beta^{-1} \diag \left (\Omegabf^{-1},\;  \(1-\deltat^{2}/4\)\M, \; \Q \right ),
\end{equation}
is invariant under $\textnormal{gle-BAOAB}$ with $\Q$ as defined in \cref{prop:GLE:cons}.
\end{theorem}
\begin{proof}
We take a dual perspective and show that the above specified Gaussian measure $\muinv_{\deltat}$ is the unique stationary solution of the corresponding forward equation, i.e.,
\begin{equation}\label{eq:forward}
 \left (  \hat{\Pc}^{{\textnormal{gle-BAOAB}}}_{\deltat} \right)^{\dagger}\muinv_{\deltat} = \muinv_{\deltat},
\end{equation}
where $  \big(  \hat{\Pc}^{{\textnormal{gle-BAOAB}}}_{\deltat} \big)^{\dagger} = \exp( \frac{\deltat}{2} \mathcal{L}_{B}^{\dagger})\exp( \frac{\deltat}{2} \mathcal{L}_{A}^{\dagger})\exp( \deltat \mathcal{L}_{O}^{\dagger})\exp( \frac{\deltat}{2}  \mathcal{L}_{A}^{\dagger})\exp( \frac{\deltat}{2}  \mathcal{L}_{B}^{\dagger})$, is the forward operator of the gle-BAOAB scheme, which by construction is simply the concatenation of the forward operators corresponding to the respective B/A/O- steps in the order given by the splitting scheme. The action of these forward operators when applied to a multivariate Gaussian measure \[\mathcal{N}(\x | \mubf, {\bm V}) \dd \x \propto \exp\left({-(\x-\mubf)^{\trans}{\bm V}^{-1}(\x-\mubf)}\right) \dd \x,\] is found to be
 \begin{align*}
\exp \left ( \frac{\deltat}{2}\mathcal{L}_{A}^{\dagger} \right)\mathcal{N}(\;\cdot\; \;| \mubf,{\bm V}) &= \mathcal{N}(\;\cdot\; \;| ~\Psibf_{A} \mubf,\Psibf_{A} {\bm V} \Psibf_{A}^{\trans} ),\\
\exp\left ( \frac{\deltat}{2} \mathcal{L}_{B}^{\dagger}\right)\mathcal{N}(\;\cdot\; \;| \mubf,{\bm V}) &= \mathcal{N}(\;\cdot\; \;| ~\Psibf_{B} \mubf,\Psibf_{B} {\bm V} \Psibf_{B}^{\trans} ),
\end{align*}
where 
\[
\Psibf_{A}=
 \begin{bmatrix}
\I_{n} & \frac{\deltat}{2} \M^{-1} & \0 \\
\0 &  \I_{n}   & \0 \\
 \0 &  \0  & \I_{m} 
\end{bmatrix} ,
\quad
\Psibf_{B}=
 \begin{bmatrix}
\I_{n}   & \0 &{\bm 0} \\
-\frac{\deltat}{2}\Omegabf &  \I_{n}  & \0 \\
 \0  &  \0 & \I_{m} 
\end{bmatrix},
\]
and 
\begin{align*}
\exp( \deltat \mathcal{L}_{O}^{\dagger})\mathcal{N}(\;\cdot\; \;| \mubf,{\bm V}) &=  \mathcal{N}(\;\cdot\; \; | \widetilde{\F}\mubf, \widetilde{\F}{\bm V} \widetilde{\F}^{\trans} + \widetilde{\S}\widetilde{\S}^{\trans}),
\end{align*}
where
\begin{equation}\label{eq:F:S}
\widetilde{\F} = 
\begin{pmatrix}
\I_{n} & \0\\
\0 & \F
\end{pmatrix},
\quad
\widetilde{\S} = 
\begin{pmatrix}
\0 & \0\\
\0 & \S
\end{pmatrix}.
\end{equation}
Thus, the Gaussian measure with density $\mathcal{N}(\;\cdot\; \;| \mubf,{\bm V})$ is invariant under the action of $\big(  \hat{\Pc}^{{\textnormal{gle-BAOAB}}}_{\deltat} \big)^{\dagger} $ exactly if 
\begin{equation*}
\begin{aligned}
\mubf &= \Psibf_{B}\Psibf_{A}  \widetilde{\F} \Psibf_{A}\Psibf_{B} \mubf,\\
{\bm V} &= \Psibf_{B}\Psibf_{A}\widetilde{\F}\Psibf_{A}\Psibf_{B}{\bm V} \Psibf^{\trans}_{B}\Psibf^{\trans}_{A}\widetilde{\F}^{\trans}\Psibf^{\trans}_{A}\Psibf^{\trans}_{B}
+ \Psibf_{B}\Psibf_{A}\widetilde{\S}\widetilde{\S}^{\trans} \Psibf^{\trans}_{A}\Psibf^{\trans}_{B}.
\end{aligned}
\end{equation*}
which is satisfied if $\mubf = \mubf_{\deltat}$ and ${\bm V} ={\bm V}_{\deltat}$ with $\mubf_{\deltat}=0$ and ${\bm V}_{\deltat}$ as specified in the proposition. Since for sufficiently small $\deltat>0$ the discretized dynamics are ergodic (see \cref{sec:num:ergodic}), this solution is also the unique solution of equation \eqref{eq:forward}. 
\end{proof}
\begin{remark}\label{rem:error:other}
With the same techniques as in the proof of \cref{prop:gaussian:error}, one may show that for the same quadratic potential function, the invariant measure of the Markov chain generated by \textrm{gle-ABOBA} is identical to the invariant measure of \textrm{gle-BAOAB}, and that the unique invariant measure of the Markov chains generated by \textrm{gle-OBABO} and \textrm{gle-OABAO} is the Gaussian measure $\mathcal{N}(\x;\mubf_{\deltat}, {\bm V}_{\deltat}) \dd \x$ with $\mubf_{\deltat} = {\bm 0}$, and ${\bm V}_{\deltat} = \beta^{-1} \diag \left ( \(1-\deltat^{2}/4\)\Omegabf^{-1},\; \M,
\; \Q \right)$.

\end{remark}
\subsubsection{Superconvergence of gle-BAOAB in the overdamped limit}\label{sec:superconv}
The gle-BAOAB scheme possesses a superconvergence property in the discrete time version of the overdamped limit (see \cref{sec:disc:overdamped}). That is, for observables $\varphi \in L^{2}(\muinv)$ which are purely functions of the position variable $\q$, the incurred discretization bias of the corresponding ergodic average when  computed using the gle-BAOAB scheme applied to rescaled process \eqref{eq:GLE:mark:rescaled} with $\lambda=1$ and $\mu_{1}=\mu_{2}=\varepsilon^{-1}$ behaves as 
\begin{equation}\label{eq:superconv:close}
\abs*{\int_{\xDomain} \varphi(\q) \muinv_{\deltat}(\dd \x)- \EE_{\muinv}\varphi } = \bigO(\epsilon \deltat^{2}) + \bigO(\deltat^{4}),
\end{equation}
as $\varepsilon\rightarrow 0$ and $\deltat \rightarrow 0$. For sufficiently small values of $\varepsilon$, the magnitude of the leading order term of the discretization bias decreases linearly in $\varepsilon$. In particular, in the limit $\varepsilon = 0$, the leading error term $\bigO(\epsilon \deltat^{2})$ in \eqref{eq:superconv:close} vanishes. This results in the discretization bias to decrease at fourth order in $h$ (instead of second order as one would expect by construction) -- a property which we refer to as ``superconvergence''.

We formally show this result for a particle of unit mass in a one dimensional positional domain and memory kernel corresponding to a matrix of the $\Gammabf$ of the generic form
\[
\Gammabf =
\begin{pmatrix}
\Gamma_{1,1} & \Gamma_{1,2}\\
\Gamma_{2,1} & \Gamma_{2,2}
\end{pmatrix} \in \RR^{2}.
\]
which is assumed to satisfy \cref{as:gle}. Our derivation can be extended to more general forms of \eqref{eq:qgle}, but we refrain from doing so in order to keep notation simple.  As a starting point of the derivation we consider again a Taylor expansion of the evolution operator 
\begin{equation}
\hat{\Pc}^{{\textrm{gle-BAOAB}}}_{\deltat} = {\rm Id} + \deltat\Acal_{1}  +  h^{3}\Acal_{3} \varphi (x) + O(h^{5}) ,
\end{equation}
where $\Acal_{1} = \Lcg$ and $\Acal_{3}$ as defined in \eqref{eq:ac2} with $X={\rm B},Y={\rm A},Z={\rm O}$. By \cite[Theorem 3.3]{Lelievre2016a} and under suitable regularity conditions on the generator $\Lcg$ and on the operators $\Acal_{k}, k \geq 1$ (see \cref{rem:expansion:reg}), there exists $\deltat^{*}>0$ so that the expectation of test functions $\varphi\in L^{2}(\muinv)$ with respect to the perturbed invariant measure $\muinv_{\deltat}$ can be expanded as
\begin{equation}\label{eq:error:expantion:general}
\int_{\xDomain} \varphi(\x) \muinv_{\deltat}(\dd \x) =  \EE_{\muinv}\varphi +  \deltat^{2}\int_{\xDomain} \varphi(\x) f_{3}(\x) \muinv (\dd \x) + \deltat^{4} R_{\varphi,\deltat }
\end{equation}
with $\abs{R_{\varphi,\deltat}}$ uniformly bounded for $\deltat\in(0,\deltat^{*}]$. The correction term $f_{3}$ is obtained as the solution of 
\begin{equation}\label{eq:pde:f3}
\Lcg^{*} f_{3}  = \Acal_{3}^{*}{\bf 1},
\end{equation}
  where the explicit form of the right hand side can be computed as
\[
\begin{aligned}
\Acal_{3}^{*}{\bf 1} &= -\frac{1}{4}\beta\left( \p^2 \Gamma_{1,1}+   \p\; \s \Gamma_{2,1}  - \beta ^{-1}\Gamma_{1,1} \right)  U''(\q)  \\
& ~~-\frac{1}{12} \beta  \p^3 U^{(3)}(\q)+\frac{1}{4} \beta  \p U'(\q) U''(\q).
\end{aligned}
\]
Here, and below, we denote $L^{2}(\muinv)$-adjoint of an operator $\Acal$ by $\Acal^{*}$ so that $\innerLmu{\Acal g,f}= \innerLmu{g,\Acal^{*}f}$ for all $g,f \in L^{2}(\muinv)$, where $\innerLmu{f,g} := \int f g \; \dd \muinv$. By virtue of the Fredholm alternative equation \eqref{eq:pde:f3} possesses a solution iff  $ \innerLmu{g,\Acal_{3}^{*}{\bf 1}} = 0$ for all functions $g$ contained in the null space of $\Lcg$. Since the SDE associated with the generator $\Lcg$ is by assumption ergodic,
the null space of $\Lcg$ only contains constant functions for which $\innerLmu{g,\Acal_{3}^{*}{\bf 1}} \propto \innerLmu{{\bf 1},\Acal_{3}^{*}{\bf 1}}  = 0$ is  indeed true. 

Finding a closed form solution of the PDE \eqref{eq:pde:f3} is still intractable for general potentials. Instead, we employ a singular perturbation approach. Under the scaling $\lambda=1, \mu_{1}=\mu_{2}=\varepsilon^{-1}$ the generator decomposes as $\Lcg = \varepsilon^{-1}\Lco + \Lch$ and we can expand the solution $f_{3}$ in $\varepsilon$ as $f_{3}= f_{3,0}+ \varepsilon f_{3,1} + \varepsilon^{2} f_{3,2} +O(\varepsilon^{3})$. By plugging this into \eqref{eq:pde:f3} we get
\begin{equation}\label{eq:exp:epsilon}
\left ( \frac{1}{\varepsilon} \Lco^{*} + \Lch^{*} \right ) \left (f_{3,0}+ \varepsilon f_{3,1} + \varepsilon^{2} f_{3,2} +O(\varepsilon^{3}) \right )  = \Acal_{3}^{*}{\bf 1}
\end{equation}
from which we obtain the following collection of PDEs by  equating powers of $\varepsilon$
\begin{align}
\Lco^{*}f_{3,0} & = \frac{1}{4} \beta  \p^2 \Gamma_{1,1} U''(\q)-\frac{1}{4} \Gamma_{1,1} U''(\q) + \frac{1}{4}\beta s^{T}  \Gamma_{2,1}  \p  U''(\q) , \label{pde1b}\\
\Lch^{*} f_{3,0} + \Lco^{*} f_{3,1}&=\frac{1}{12} \beta  \p^3 U^{(3)}(\q)-\frac{1}{4} \beta  \p U'(\q) U''(\q) , \label{pde2b}\\
\Lch^{*} f_{3,1} + \Lco^{*} f_{3,2}&= 0,\label{pde3b}\\
\Lch^{*} f_{3,i} + \Lco^{*} f_{3,i+1}&= 0, ~ i \geq 2.\label{pde4b}
\end{align}
Solving this system iteratively, we find (see \cref{supp:sec:pde:solution} for details):
\[
 f_{3,0}(\q,\p,\s) = -\frac{1}{8 }\beta  \p^2 U''(\q) +  \frac{1}{8}U''(\q),
\]
which can be verified to satisfy $\int_{\xDomain } \varphi(\q) f_{3,0}(\x) \muinv( \dd \x) = 0$, for any observable $\varphi\in \testfunc$ which is purely a function of $\q$. Thus, for such $\varphi$, \eqref{eq:error:expantion:general} can be written as 
\[
\begin{aligned}
\int_{\xDomain} \varphi(\q) \muinv_{\deltat}(\x) \dd \x 
=&\EE_{\muinv}\varphi + \epsilon \deltat^{2}\int_{\xDomain} \varphi(\q)  f_{3,1}(\q)\muinv(\dd \x) +\bigO(\epsilon^{2}\deltat^{2})   + \bigO(\deltat^{4}),
\end{aligned}
\] 
as $\varepsilon\rightarrow 0$ and $\deltat \rightarrow 0$, which is the desired statement.
\begin{remark}\label{rem:expansion:reg}
The formal error analysis can be made rigorous by showing that the remainder terms in expansions \eqref{eq:error:expantion:general} and \eqref{eq:exp:epsilon} are uniformly bounded for sufficient small $\deltat$, and $\varepsilon$, respectively. For the expansion \eqref{eq:error:expantion:general} it would be sufficient to show that the conditions of  \cite[Theorem 3.3]{Lelievre2016a} are indeed satisfied. In particular, this would entail showing that the function set $\mathcal{C}_{P,0}^{\infty}(\xDomain,\RR):= \left \{ \varphi \in \testfunc : \int \varphi\, \dd \muinv = 0 \right \}$,
is invariant under application of the operators $\Lcg^{-1}$ and $\left( \Lcg^{*} \right )^{-1}$, as well as that $\testfunc$ is invariant under application of the operators $\Acal_{k}, k\in \NN$. Analogous estimates have been shown in \cite{Kopec,Kopeca,redon2016error} for the generators of the overdamped Langevin equation, the underdamped Langevin equation, and Langevin equations with generalized kinetic energies, respectively. Moreover, in order to make the expansion of \eqref{eq:exp:epsilon} rigorous one would need to show --as in \cite{LeMaSt2015}-- a uniform Hypocoercivity property of the form: there is a $K>0$ such that $\norm{(\varepsilon^{-1}\Lco + \Lch)^{-1}\varphi }_{H^{1}(\muinv)} \leq K \norm{\varphi}_{H^{1}(\muinv)}$
for any $\varepsilon>0$ and all test functions $\varphi$ contained in the weighted Sobolev space $H^{1}(\muinv)$ and which are such that for almost all $\q$ the mean with respect to the marginal measure $\muinv(\dd \p \, \dd \s)$ vanishes.
\end{remark}

\section{White noise and overdamped limit of the gle-BAOAB method}\label{sec:limit:methods}
In this section we analyze the behavior of the gle-BAOAB splitting method in the overdamped and white noise limit discussed in \cref{sec:limit:dyn}. 
For this purpose consider the stochastic flow-map of the gle-BAOAB method when applied to the rescaled process \eqref{eq:GLE:mark:rescaled},
\begin{equation}
\hat{\Phi}^{{\textrm{gle-BAOAB}}}_{\deltat,\mu} = \Phi^{\rm B}_{\deltat/2} \circ \Phi^{\rm A}_{\deltat/2} \circ \Phi^{\rm O}_{\deltat,\mu} \circ \Phi^{\rm A}_{\deltat/2} \circ \Phi^{\rm B}_{\deltat/2},
\end{equation}
where 
\begin{equation}
\Phi^{\rm O}_{\deltat,\mu}: (\q,\p,\s) \mapsto  (\q, {\bm F}^{\mu}_{\deltat}\; (\p, \s)^{\trans} + {\bm S}^{\mu}_{\deltat} \Rand), \; \Rand \sim \mathcal{N}( {\bm 0}, \I_{n+m}),
\end{equation}
with 
\[
{\bm F}^{\mu}_{\deltat} := \exp (-\deltat \Gammabf^{\mu}), \quad 
\left ( {\bm S}^{\mu}_{\deltat}\right ) ^{\trans}{\bm S}^{\mu}_{\deltat} = 
\begin{pmatrix}
\I_{n} & \0 \\
\0 & \Q
\end{pmatrix} - {\bm F}^{\mu}_{\deltat}\begin{pmatrix}
\I_{n} & \0 \\
\0 & \Q
\end{pmatrix} { {\bm F}^{\mu\trans}_{\deltat}},
\] 
and $\Gammabf^{\mu}$ as defined in \eqref{eq:Gamma:rescale}. In both limits \textrm{gle-BAOAB} converge to state-of-the-art numerical integration schemes for the corresponding limiting dynamics which have been shown exhibit particularly low discretization bias (see \cite{Leimkuhler2013a,LeMaSt2015}). 
\subsection{White noise limit}
The \textup{gle-BAOAB} integration scheme, when applied to the rescaled process \eqref{eq:GLE:mark:rescaled} with  $\lambda = 1, \mu_{1} =\epsilon^{-1},\; \mu_{2}=\epsilon^{-2}$, reduces to the \textup{ld-BAOAB} discretization of an underdamped Langevin equation in the white noise limit $\epsilon \rightarrow 0$.
More precisely,  we have the following result:
\begin{theorem}[White noise limit of gle-BAOAB]\label{col:discrete:WN}
Let $(\q_{k},\p_{k},\s_{k})_{k\in \NN}$ be the Markov chain obtained by with the \textup{gle-BAOAB} method $\hat{\Phi}^{\textup{gle-BOAOB}}_{\deltat,(\epsilon^{-1},\epsilon^{-2})}$. Let $(\hat{\q}_{k},\hat{\p}_{k})_{k\in \NN}$ denote the Markov chain generated by the \textup{ld-BAOAB} method of \cite{Leimkuhler2013a} (see also {\em Algorithm} \ref{alg:ld:baoab})
when applied to \eqref{eq:LD} with friction tensor $\widehat{\Gammabf} = {\bm D}_{a}^{2}{\bm D}_{b}^{-1}$, diffusion tensor $\widehat{\Sigmabf}= \sqrt{2}{\bm D}_{a}{\bm D}_{b}^{-1/2}$ and stepsize $\deltat$. Then, for all $N\in \NN$, we have
\[
(\q_{k},\p_{k})_{0\leq k \leq N}  \xrightarrow[\varepsilon \to 0]{\mathrm{law}} (\hat{\q}_{k},\hat{\p}_{k})_{0\leq k\leq N}.
\]
\end{theorem}
\begin{proof}
Since $(\hat{\q}_{k},\hat{\p}_{k})_{k\in \NN}$ is a Markov processes, it is sufficient to show that the transition probabilities converge appropriately, i.e., 
\[
\Pi \left (\hat{\Phi}^{\textup{gle-BAOAB}}_{\deltat,(\varepsilon^{-1},\varepsilon^{-2})}(q,p,s) \right)  \xrightarrow[\varepsilon \to 0]{\mathrm{law}} \hat{\Phi}^{\textup{ld-BAOAB}}_{\deltat}(q,p),
\]
for all  $(q,p,s) \in \xDomain$, where $\hat{\Phi}^{{\textrm{ld-BAOAB}}}_{\deltat,\mu}$ denotes the stochastic flow map of the \textup{ld-BAOAB} splitting scheme, and $\Pi  :\,(q,p,s) \mapsto (q,p)$ denotes the projection operator on the position and momentum component. 

The two methods only differ in terms of their respective O-steps. It is therefore sufficient to show that in the limit $\varepsilon \rightarrow 0$, these become identical in distribution, which is exactly the case if
\[
\lim_{\epsilon \rightarrow 0} \exp \left (-\deltat \Gammabf^{(\epsilon^{-1},\epsilon^{-2})}  \right )
= \begin{pmatrix} \exp( -\deltat {\bm D}_{a}^{2}{\bm D}_{b}^{-1}) & \0 \\ \0 & \0 \end{pmatrix}.
\]
where \[
\Gammabf^{(\epsilon^{-1},\epsilon^{-2})}=\begin{pmatrix}
\0 &  -  \epsilon^{-1}{\bm D}_{a}\\
  \epsilon^{-1} {\bm D}_{a} &  \epsilon^{-2}{\bm D}_{b}
\end{pmatrix} .
\]
We show this by applying a suitable similarity transformation:  without loss of generality let  $\deltat = 1$, and consider the orthogonal matrix,
\[
{\bf O}=\hat{\I}_{2n-1,2n}^{(2n)}\hat{\I}_{2n-3,2n-1}^{(2n)}\hdots  \hat{\I}_{5,n+3}^{(2n)}  \hat{\I}_{3,n+2}^{(2n)}\hat{\I}_{1,n+1}^{(2n)},
\]
where $\hat{\I}_{i,j}^{(2n)}$ denotes the elementary matrix whose action when multiplied from the left to a matrix ${\bm A}\in \RR^{2n\times 2n}$ corresponds to a swap of $i$-th and $j$-th rows of ${\bm A}$, so that 
\[
\Gammabf^{(\epsilon^{-1},\epsilon^{-2})}
=
{\bf O}^{\trans}
\diag \left ( 
A_{1}^{\epsilon},A_{2}^{\epsilon},\dots, A_{n}^{\epsilon} \right ) 
{\bf O},
\quad 
\text{with} ~
A_{i}^{\epsilon}  =\begin{pmatrix} 0 & - \epsilon^{-1}a_{i} \\ \epsilon^{-1}a_{i} & \epsilon^{-2}b_{i}\end{pmatrix}, ~ 1\leq i \leq n.
\]
By \cref{lem:exp:limit} we have $\lim_{\epsilon \rightarrow 0} A_{i}^{\epsilon} = B_{i}$ with $B_{i}= \begin{pmatrix} e^{-a_{i}^{2}/b_{i}} & 0 \\ 0 & 0 \end{pmatrix}$. Thus, 
\[
\begin{aligned}
\lim_{\epsilon\rightarrow 0}\exp(-\Gammabf^{(\epsilon^{-1},\epsilon^{-2})} ) &= \lim_{\epsilon\rightarrow 0} {\bf O}^{\trans}
 \diag \left (\exp(-A_{1}^{\epsilon}),\dots, \exp(-A_{n}^{\epsilon}) \right )   {\bf O}\\
 &=
 {\bf O}^{\trans} \diag \left ( \lim_{\epsilon\rightarrow 0} \exp(-A_{1}^{\epsilon}),\dots,  \lim_{\epsilon\rightarrow 0}\exp(-A_{n}^{\epsilon}) \right )   {\bf O}\\
 &=
 {\bf O}^{\trans} \diag \left (B_{1},
 \dots, B_{n}\right )   {\bf O}
= \begin{pmatrix} \exp( -\deltat {\bm D}_{a}^{2}{\bm D}_{b}^{-1}) & \0 \\ \0 & \0 \end{pmatrix}.
\end{aligned}
\]
\end{proof}

\subsection{Overdamped limit}\label{sec:disc:overdamped}
When applied to the rescaled process \eqref{eq:GLE:mark:rescaled} with  $\lambda = 1, \mu_{2} =\epsilon^{-1},\; \mu_{3}=\epsilon^{-1}$ the gle-BAOAB method reduces to the BAOAB-limit method (``Leimkuhler-Matthews method'') of \cite{Leimkuhler2013a} in the asymptotic limit $\epsilon \rightarrow 0$ as shown in the following:
\begin{theorem}\label{prop:method:od:limit}
Let $(\q_{k},\p_{k},\s_{k})_{k\in \NN}$ be the Markov chain obtained by with the \textup{gle-BAOAB} method $\hat{\Phi}^{\textup{gle-BOAOB}}_{\deltat,(\epsilon^{-1},\epsilon^{-1})}$ with $\p_{0} \sim \mathcal{N}(\0,\I_{n})$. Let $(\widetilde{\q}_{k})_{k\in \NN}$ denote the Markov chain generated by the BAOAB-limit method,
\begin{equation}\label{alg:lm}
\widetilde{\q}_{k+1} \gets  \widetilde{\q}_{k} - \tilde{\deltat} \Lambda \nabla U(\q_{k}) +\sqrt{2 \tilde{\deltat} \Lambda} \frac{1}{2}(\widetilde{\Rand}_{k} +\widetilde{\Rand}_{k+1} ),
\end{equation}
 with $ \widetilde{\q}_{0}$ and $\q_{0}$ being identically distributed, $\widetilde{\Rand}_{k} \sim \mathcal{N}( \0,\I_{n}),\, k\in \NN$ independent, and stepsize $\tilde{\deltat} = \deltat^{2}/2$, and $\Lambda =\I_{n}$. Then, for all $N\in \NN$, we have
\[
(\q_{k})_{0\leq k \leq N}  \xrightarrow[\varepsilon \to 0]{\mathrm{law}} (\hat{\q}_{k})_{0\leq k\leq N}.
\]
\end{theorem}
\begin{proof}
For $\mu=(\varepsilon,\varepsilon^{-1})$, we have $\Gammabf^{\mu} = \varepsilon^{-1}\begin{pmatrix} 
\0 & -{\bm D}_{a}\\
{\bm D}_{a}&{\bm D}_{b}
\end{pmatrix}
$, thus,
\[
\F=  \exp \left ({-\frac{\deltat}{\epsilon}\begin{pmatrix} 
\0 & -{\bm D}_{a}\\
{\bm D}_{a}&{\bm D}_{b}
\end{pmatrix}} \right )\rightarrow {\bm 0} \quad  \text{as} \quad  \epsilon \rightarrow 0,
\]
 and therefore also $\S\rightarrow \beta^{-1/2} 
\I_{n+m}, \text{ as }\;\epsilon \rightarrow 0$.
Thus, 
\[
\lim_{\epsilon\rightarrow 0}\Phi^{\rm O}_{\deltat,(\epsilon^{-1},\epsilon^{-1})}(\q,\p,\s)= (\q, \beta^{-1/2} \Rand), \; \Rand \sim \mathcal{N}( {\bm 0}, \I_{n+m}),
\]
which removes any coupling between the auxiliary variable $\s$ and $(\q,\p)$. Consequently, in the limit of $\varepsilon \rightarrow 0$ we can disregard the $\s$-component in the corresponding updating sequence of the positions and momenta. Moreover, since the momentum variables are independently resampled at every iteration, we can eliminate the momentum component from the  updating sequence to obtain \eqref{alg:lm} with $\Lambda = \I_{n}$, $\tilde{\deltat} = \deltat^{2}/2$, $\Rand_{-1} = \p_{0}$.
\end{proof}

\section{Numerical experiments}\label{sec:numExp}
In this section we  assess the performance of the splitting methods which we introduced in \cref{sec:numint} in numerical experiments. 
\subsection{Comparison of proposed splitting schemes}\label{sec:comp:HOU:methods}
We first compare the performance of the methods discussed in this article against each other. For this purpose we consider a simple QGLE on a one-dimensional positional domain with potential function
\begin{equation}\label{dwpotential}
U_{DW}(\q) =  \frac{1}{2}\q^{2} + \sin(1/4 + 2 \q), 
\end{equation}
which is  an uneven double-well. We evaluate the performance in terms of the incurred stepsize-dependent discretization bias for observables which are purely functions of the position variable. For the parameterization of the noise process in the GLE we consider the memory kernels
\begin{align}\label{kernel:prony1}
\K (t) = 2^{r}K(t 2^{r}), \quad K(t) := \frac{5}{2} \exp(-t/4) + \frac{1}{2}\exp(-t/8),
\end{align}
where we let $r$ take values in $\{0,1,2\}$. The $r$-dependent parameterization of the memory kernels is chosen such that in the limit of $r \rightarrow \infty$, the corresponding GLE approaches an underdamped Langevin equation.

We consider as an error measure 
\[
\error_{\rm MAE}((\q_{k})_{1\leq k\leq N}) 
= n_{B}^{-1}\sum_{i=1}^{n_{B}}\abs*{\left (  \frac{1}{N} \sum_{k=0}^{N-1} \mathbbm{1}_{B_{i}}(\q_{k}) - \int_{\xDomain} \mathbbm{1}_{B_{i}}(\q) \muinv(\dd \x)  \right )},
\]
where the equal sized bins $B_{i} \subset \RR ,i=1,\dots,n_{B}$ are chosen such that they form a partition of an interval $[a,b] \subset \RR$, which contains $99.99\%$ of the probability mass of the Gibbs measure associated with $\RR$. The quantity $\error_{\rm MAE}((\q_{k})_{1\leq k\leq N})$ may be considered as the mean approximate error (MAE) of the discretization bias incurred for the observables $\varphi_{B_{i}}: \q \mapsto  \mathbbm{1}_{B_{i}}(\q), i=1,\dots,n_{B}$, or, as an estimate of the total variation distance between the perturbed invariant measure $\muinv_{\deltat}$ and the exact target measure $\muinv$. 

In total, $100$ trajectories, all initialized in accordance with the exact equilibrium distribution $\muinv$, were simulated over a physical time period of length $T=\deltat N = 10^{7}$ to obtain the statistics.

Figure \ref{plot:comp:new} shows $\error_{\rm MAE}$ for the splitting schemes discussed in \cref{sec:numint}. All methods displayed are by construction second order. Differences in performance are thus measured in terms of the magnitude of the corresponding pre-factors of the leading error term. We find that the discretization error incurred in \textup{gle-OBABO} and \textup{gle-OABAO} is comparable and is not noticeably affected by the parameterization of the memory kernel. In comparison to that, the discretization error of \textup{gle-BAOAB} and \textup{gle-ABOBA} is smaller, and the accuracy of \textup{gle-BAOAB} improves significantly with increasing value of $r$ in the parameterization of the memory kernel. 
\begin{figure}[ht]
\hspace{-.5cm}
\includegraphics[width=1\textwidth]{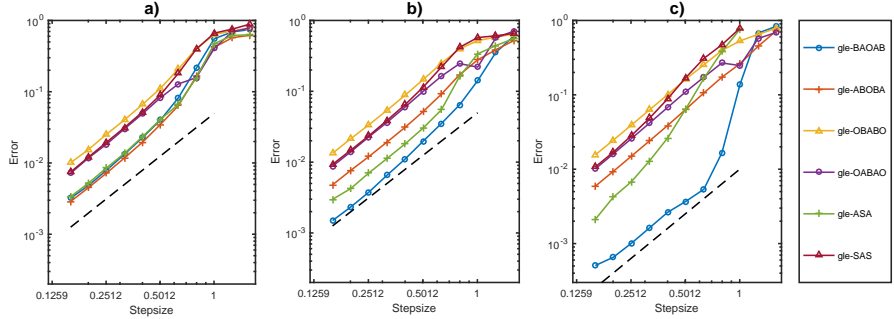}
\caption[]{Log-Log plot of the MAE of numerical integrators proposed in this article. Results for the GLE with potential function \eqref{dwpotential} and memory kernel \eqref{kernel:prony1} with $r=0,1,2$, are shown in panel a), b) and c), respectively.  Any missing error value indicates numerical instability of the respective method for the corresponding stepsize. The dashed black line corresponds to a second order decay. Details on the integrators \textup{gle-ASA} and \textup{gle-SAS} can be found in \cref{sec:construction:stochastic-verlet}. 
}\label{plot:comp:new}
\end{figure}

\subsection{Comparison with previously proposed GLE schemes}\label{sec:comp:methods}
We next compare the performance of gle-BAOAB with methods previously proposed in the literature using the same setup as in \cref{sec:comp:HOU:methods}. We compare the gle-BAOAB method with the methods proposed in \cite{Baczewski2013a} (BB-BAOB, BB-BACOCAB), \cite{stella2014generalized} (KLS-OBOAB), and \cite{Ceriotti2010} (gle-OBABO). These methods are all constructed as weak second-order schemes.

For moderate variance and slowly decaying autocorrelation of the noise-process (that is $r=0,1$) we observe that the error incurred by the methods BB-BACOCAB and KLS-OBABO is very similar to the error  of the {\rm gle-BAOAB} method (Figure \ref{plot:comp:other}, a,b). For all choices of the memory kernel, the error in gle-OBABO and BB-BAOB is at least by a factor of 10 higher than the error of gle-BAOAB, and this factor increases further with increasing value of $r$. Similarly, with an increasing value of $r$, the accuracy of the KLS-OBABO method decreases and comparison of gle-BAOAB, and the maximum admissible stepsize of BB-BAOB and BB-BACOCAB decreases significantly, while the maximum admissible step size for schemes discussed in \cref{sec:numint} is not affected. 
The high accuracy of BB-BACOCAB is not surprising as the authors in \cite{Baczewski2013a} specifically design this method for the sampling of accurate configurational averages. Interestingly, the KLS-OBABO method exhibits comparable accuracy even though the construction of this numerical scheme is not based on a systematic analysis of the discretization error in configurational averages. Finally, it is important to note that the scope of the memory kernels to which the methods proposed in \cite{Baczewski2013a} are applicable is very limited in comparison to the class of memory kernels which can be simulated using gle-BAOAB and gle-OBABO.
\begin{figure}[ht]
\hspace{-.5cm}
\includegraphics[width=1\textwidth]{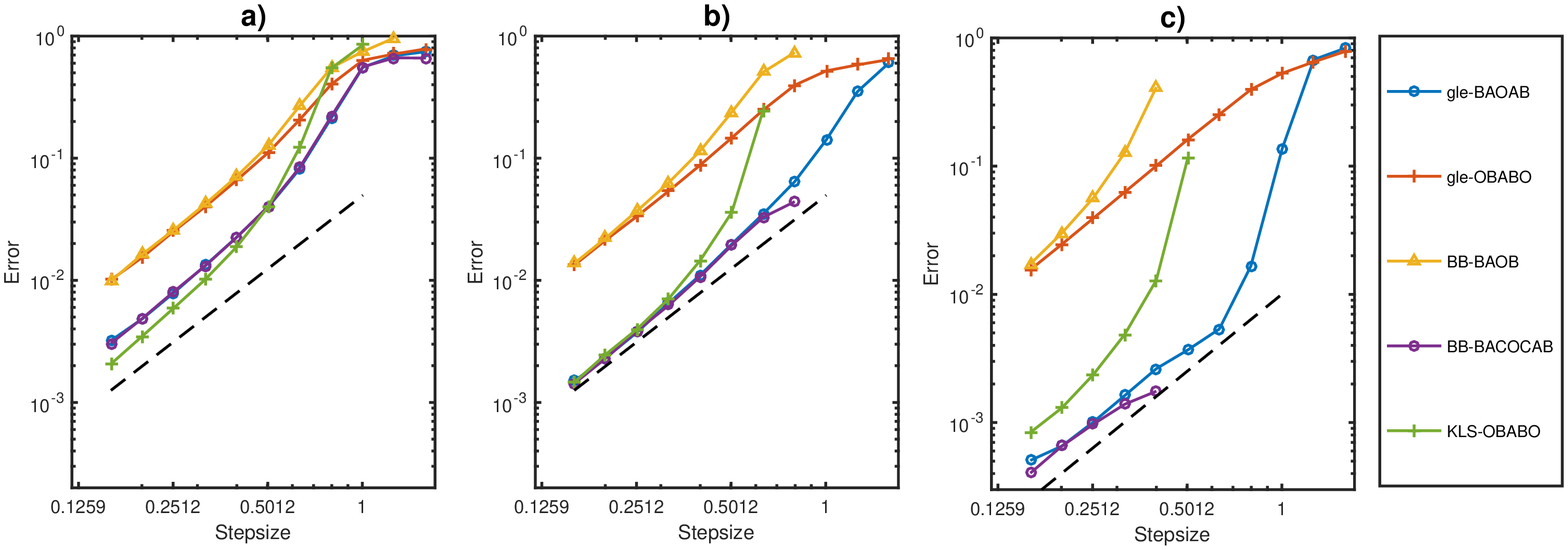}
\caption[]{Different integrators but otherwise same setup as in \cref{plot:comp:new}. 
}\label{plot:comp:other}
\end{figure}

\subsection{Parameter-dependent accuracy of {gle-BAOAB}}
In order to support the results derived by the singular perturbation ansatz in  \cref{sec:superconv}, we evaluate the sampling accuracy of {\rm gle-BAOAB} when applied to a GLE with a simple exponentially decaying memory kernel, i.e., 
\begin{equation}
\K(t) = \gamma e^{-t/\tau},\gamma>0,\tau>0,
\end{equation}
and the potential function \eqref{dwpotential}. 
As predicted we find that the discretization bias decreases as the overdamped limit is approached  (See \cref{plot:comp:params} b). Moreover, for parameter values $\lambda = 128, \tau = 1/16$, we find the predicted 4th order decay of the discretization bias as $\deltat$ tends to $0$. For the chosen range of parameter values we further observe (i) a decrease of the MAE in the white noise limit (\cref{plot:comp:params} {\textbf a)}.), (ii) a decrease of the MAE for fixed decay rate $\tau=1$ as the pre-factor $\gamma$ increases  (\cref{plot:comp:params} {\textbf d)}.), (iii) no systematic change of the magnitude of the MAE  for fixed pre-factor $\gamma=4$ and varying decay rate $\tau$  (\cref{plot:comp:params} {\textbf c)}.).
\begin{figure}[ht]
\includegraphics[width=.8\textwidth]{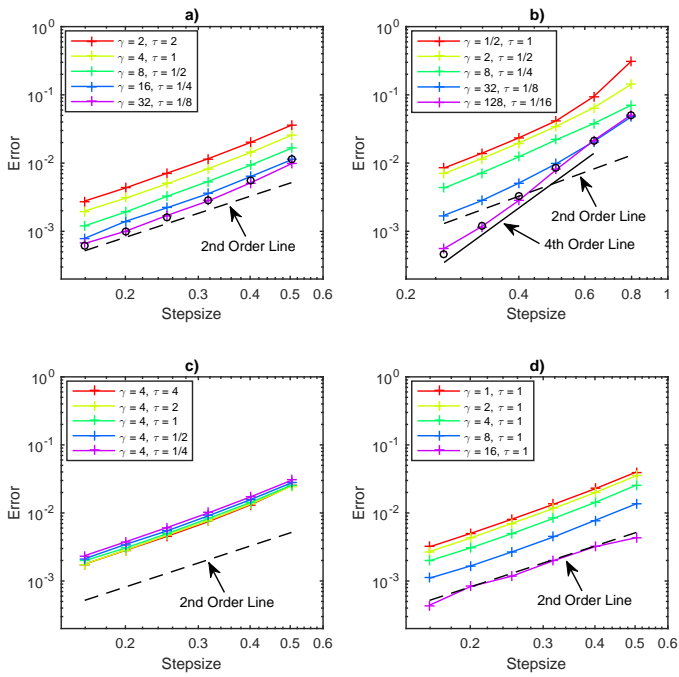}
\caption[]{Log-Log plot of stepsize vs. MAE of {\rm gle-BAOAB} applied to the GLE with potential function \eqref{dwpotential} and memory kernel $K(t)= \gamma e^{-t/\tau}$. Black circles show the observed MAE of the respective limiting dynamics. 
}\label{plot:comp:params}
\end{figure}

\subsection{Application to Bayesian posterior sampling}\label{sec:Hidalgo}
In this section we consider Bayesian Gaussian mixture model as a simple application of the discussed sampling methods and we use this application to demonstrate how the gle-BAOAB method when used in combination with the GLE-dynamics developed in \cite{Ceriotti2010} results in a sampling scheme which has drastically improved sampling properties in comparison to BAOAB discretizations of the underdamped Langevin equation as well as in comparison to the sampling scheme proposed in the above mentioned reference, which in the language used in this article corresponds to the \textup{gle-OBABO} method.

As a benchmark system we consider a Bayesian Gaussian mixture model applied to the Hidalgo stamp dataset \cite{Izenman1988}, which consists of the measurements $\{x_{i}\}_{i=1}^{N} \subset \RR$ of the thickness of $N=482$ postage stamps from the 1872 Hidalgo issue of post stamps.  We parametrize the model similarly to as described in \cite{Stoltz2012}; (See also \cite{Richardson1997} and \cite{Jasra2005}). That is, we choose the number of components as $N_{c}=3$ and assume isotropic Gaussian components resulting in a parameter vector $\q = ( \left ( {\bf w}_{k})_{1\leq k \leq 3}, (\mubf_{k}, {\bm \lambda}_{k}\right )_{1\leq k \leq 3}, \beta)\in {\bf \Delta}^{3}\times \RR^{6} \times \RR$,
where ${\bf \Delta}^{3}$ denotes the standard simplex in $\RR^{3}$, and ${\bf w}_{k}$ is the weight parameter, and $\mubf_{k}, {\bm \lambda}_{k}$ the mean and precision of the $k$th Gaussian component, respectively, and $\beta\in \RR$ denotes an additional hyper-parameter of the prior distribution.  The resulting target distribution is then given as the Gibbs measure of the corresponding negative log-posterior function
\begin{equation}
U(\q) = 
-\sum_{i=1}^{N}  \log p \left (  x_{i} \given \q \right  )  - \log p_{\rm prior}(\q),
\end{equation}
where the exact form of the likelihood function $p(x_{i} \given \q )$ and the prior $p_{\rm prior}(\q)$ are both specified in \cref{supp:sec:hidalgo}. 

We parametrize both the {\rm gle-BAOAB} scheme and the {\rm gle-OBABO} scheme with the pre-optimized memory kernel \textup{kv-8-8} obtained from the website GLE4MD \cite{gle4md} (see also \cref{sec:opt:memory:kernel}). 
We compare the performance of the sampling schemes 
\begin{enumerate}[label=(\roman*)]
\item   in terms of the observed discretization bias which we measure by the relative error incurred for the variable specific configurational temperatures 
\[
\varphi_{CT,i}(\q) = \q_{i}\partial_{\q_{i}}U(\q), \quad i=1,\dots,9.
\]
\item in terms of mixing which we measure by estimates of the integrated autocorrelation times  
\[
\tau_{{i}} = \int_{0}^{\infty} \EE \left [ (\q_{i}(t) - \mu_{\q_{i}}) ( \q_{i}(0)-\mu_{\q_{i}}) \right ]  \dd t, \quad i=1,\dots,9,
\]
where $\q(0)\sim \muinv$, and $\EE[\cdot]$ is the expectation with respect to $\q(0)$ and the Wiener process $\W$ in \eqref{eq:qgle}. 
\end{enumerate}
For \textup{ld-BAOAB} we considered the commonly used parameterization with a  single scalar friction coefficient, i.e., $\GammaLD = \gamma \I_{n}$. The simulation run corresponding to the parameter values $\gamma=1.0, \deltat = .01$ was obtained as the result of minimizing the integrated autocorrelation time for the slowest parameters  by varying the stepsize after fixing the friction coefficient to $\gamma=1.0$. The simulation run corresponding to the parameter values $\gamma=.1$ with $\deltat = .01$ was obtained as the result of minimizing the integrated autocorrelation time for the ``slowest parameter'' (i.e., the parameter with the largest associated integrated autocorrelation time) by simultaneously optimizing both the stepsize as well as the friction coefficient $\gamma$. The results reported for \textup{gle-BAOAB} and \textup{gle-OBABO} use a stepsize $\deltat = .02$, which was determined approximately as the maximum admissible stepsize with a few (short!) test runs. 
We find that in terms of sampling efficiency which we measured in terms of the integrated autocorrelation time of the ``slowest'' sampled parameter $\lambda_{1}$, the GLE schemes clearly outperform these as \cref{fig:hidalgo} shows. Between the GLE schemes we find that the discretization error in the sample obtained from gle-BAOAB is significantly smaller than the discretization error in the sample obtained with \textup{gle-OBABO}. The improvement in terms of the maximum admissible stable stepsize of the GLE methods in comparison to the Langevin schemes is an interesting feature.  Presumably, this is due to resonance effects which occur in the discretized dynamics of the underdamped Langevin due to an insufficient damping of fast frequency modes for the tuned value of the friction coefficient.
\begin{figure}[ht]
\includegraphics[width=1.0\textwidth]{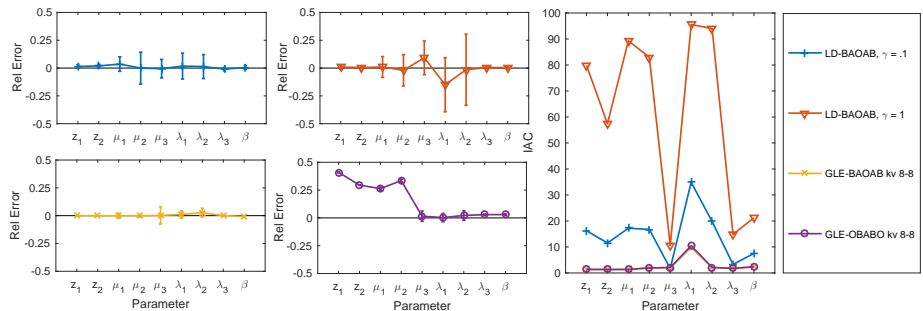}
\caption[]{Sampling statistics for the numerical experiments performed on the Hidalgo stamp dataset. The rightmost plot shows the integrated auto-correlation time (labeled as IAC) for each sampled parameter of the Gaussian mixture model. The remaining four plots on the left shows the discretization bias incurred for the variable specific configurational temperature. 
}\label{fig:hidalgo}
\end{figure}

\section*{Acknowledgments}
Both authors thank Gabriel Stoltz for helpful discussions. The research of B. Leimkuhler and M. Sachs was supported by the ERC project RULE (grant number 320823).  Prof. Leimkuhler was further supported by the Engineering and Physical Sciences Research Council under grant EPSRC EP/P006175/1 "Data-Driven Coarse-Graining using Space-Time Diffusion Maps". The work of M. Sachs was supported by the Statistical and Applied Mathematical Sciences Institute (grant DMS-1638521) and Duke University.

\bibliographystyle{siamplain}
\bibliography{refs5-2.bib}

\newpage
\appendix

\def\proj{P}

\section{Proof of \cref{prop:ergodicity:discrete}}\label{sec:proofs}
We can write the recurrence equation of the discrete gle-BAOAB solution $(\x_{k})_{k\in \NN} = (\q_{k},\p_{k},\z_{k})_{k \in \NN}$ of \eqref{eq:qgle} as

\begin{equation}\label{eq:recursion}
\begin{aligned}
\x_{k+1} 
 &=  \Psi_{{\rm B}}\Psi_{{\rm A}} \widetilde{\F}  \Psi_{{\rm A}} \Psi_{{\rm B}}\x_{k} +  \Psi_{{\rm B}} \Psi_{{\rm A}} \widetilde{\S} \begin{pmatrix} \0 \\ \Rand_{k} \end{pmatrix}
  +   \deltat R(\q_{k},\q_{k+1}),
\end{aligned}
\end{equation}
where $\Psi_{{\rm A}}, \Psi_{{\rm B}}, \widetilde{\F}, \widetilde{\S}$ are $\deltat$-dependent matrices as specified in \eqref{eq:F:S}, and 
\[
R(\q,\q') = \begin{pmatrix}
\frac{h}{2} \M^{-1}   \left [ - \M^{-1}\nabla U_{2}(\q) - (\I_{n},\0)\F \begin{pmatrix} \nabla U_{2}(\q)\\ \0 \end{pmatrix}\right ] \\
\frac{1}{2} \left [ - \nabla U_{2}(\q) - \nabla U_{2}(\q') \right ] \\
\0
\end{pmatrix}.
\]
\subsection{Lyapunov condition}\label{subsec:lya}
Recall the definition of $\Kc_{l}$  in \eqref{eq:def:lya} of \cref{prop:ergodicity:discrete} as 
\[
\Kc_{l}(x) = 1 + \norm{\x}_{\C}^{2l},\quad  l\in \NN, \qquad  \text{with} \quad \norm{\x}_{\C}^{2l}:= \left ( \x^{\trans} \C \x  \right )^{l}.
\]
Here we show that $\Kc_{l}, ~l\in \NN$ is indeed a family of Lyapunov functions satisfying \cref{as:ulc} provided that $\C$ is chosen such that the matrix $\G^{\trans} \C + \C\G$ is symmetric positive definite, where
\[
\G  = \begin{pmatrix} \0 & -\M^{-1} & \0 \\ \Omegabf & \Gammabf_{1,1} & \Gammabf_{1,2} \\   \0 & \Gammabf_{2,1} & \Gammabf_{2,2}\end{pmatrix},
\]
denotes the linear component of the drift term of \eqref{eq:qgle}. The existence of such a matrix $\C$ is equivalent to $-\G$ being a stable matrix. The latter follows as a special case of \cite[Lemma 3]{leimkuhler2017ergodic} under our assumption that $\Gammabf_{1,1}$ and $\M$ commute. 

To show this, it is sufficient to show the existence of suitable constants $a_{l}>0, b_{l}\in \RR$, so that 
\begin{equation}\label{eq:norm:Lya}
\forall \;  \x\in \xDomain,~ \forall \;  \deltat \in [0,\deltat^{*})\qquad \hat{\Pc}^{{\textnormal{gle-BAOAB}}}_{\deltat} \norm{\x}_{\C}^{2l} \leq e^{-a_{l}h}   \norm{\x}_{\C}^{2l}  + b_{l}, 
\end{equation}
for sufficiently small $\deltat^{*}>0$. This can be seen since \cref{eq:norm:Lya} implies
\[
\begin{aligned}
\hat{\Pc}^{{\textnormal{gle-BAOAB}}}_{\deltat}\Kc_{l}(x)&= 1 + \hat{\Pc}^{{\textnormal{gle-BAOAB}}}_{\deltat}\norm{\x}^{2l}_{\C}
\leq 1 + e^{-\deltat a_{l}} \norm{\x}^{2l}_{\C} + \deltat b_{l}\\
& \leq 1 - e^{-\deltat a_{l}}  +  e^{-\deltat a_{l} }\Kc_{l}(x) + \deltat b_{l}\\
& \leq  e^{-\deltat a_{l} }\Kc_{l}(x) + \deltat (b_{l} + 2 a_{l}),
\end{aligned}
\]
for all $x\in \xDomain, \deltat \in [0,\deltat^{*})$ for sufficiently small $\deltat^{*}>0$.

Central to showing the Lyapunov condition  \cref{eq:norm:Lya} are the following two Lemmas. \cref{lem:E:contraction} shows that $\Ed = \Psi_{{\rm B}}\Psi_{{\rm A}} \widetilde{\F}  \Psi_{{\rm A}} \Psi_{{\rm B}}$ possesses a spectral gap and defines a contraction in $\norm{\cdot}_{\C}$, \cref{lem:S:linear} ensures that the remaining terms in \eqref{eq:recursion} are $O(\deltat)$ as $\deltat \rightarrow 0$.



\begin{lemma}\label{lem:E:contraction}
There exists $\deltat^{*}>0$ and a constant $\lambda'>0$ so that 
\[
\forall \deltat \in [0,\deltat^{*}) : \quad \norm{\Ed}_{\mathcal{B}(\C)} \leq \exp(-\deltat (\lambda +\kappa) ),
\]
where $\norm{\cdot}_{\mathcal{B}(\C)}$ is the operator norm induced by $\norm{\cdot}_{\C}$, i.e., for a matrix $A\in \RR^{(2n+m)\times(2n+m) } $, we have
\[
\norm{A}_{\mathcal{B}(\C)} := \sup_{\x\in \RR^{2n+m}} \frac{\norm{Ax}_{\C}}{\norm{x}_{\C}}.
\]
\end{lemma}
\begin{proof}
The result can be shown using the fact that $\Ed$ is a symmetric splitting of the matrix exponential $\Ec = \exp(-h\G)$ and $\Ec$ being contracting in $\norm{\ccdot}_{\C}$. The latter follows because
\begin{equation*}
\begin{aligned}
\partial_{\deltat}\x^{\trans} \Ec^{\trans}\C\Ec \x &=  -\x^{\trans} \Ec^{\trans}[\G^{\trans}\C + \C \G] \Ec  \x \\
&\leq - \frac{\sigma_{\min}[ \G^{\trans}\C + \C \G] }{\sigma_{\max}(\C)} \x^{\trans}\Ec^{\trans}\C\Ec \x,
\end{aligned}
\end{equation*}
which implies 
\begin{equation}\label{eq:Ec:contract}
\norm{\exp(-h\G)\x}_{\C}^{2} \leq \exp(-h\lambda) \norm{\x}_{\C}^{2}, \quad \lambda = \sigma_{\min}(\G^{\trans}\C + \C\G)/ \sigma_{\max}(\C),
\end{equation}
for all $h>0$ and all $\x \in \RR^{2n+m}$ by Gr\"{o}nwall's inequality. Now, rewriting $\Ed$ as
\begin{equation*}
\Ed = \exp\left ( -\frac{\deltat}{2} {\bf B} \right ) \exp\left ( -\frac{\deltat}{2} {\bf A} \right )  \exp \left (-\deltat  \begin{pmatrix} \0 & \0 \\ \0 & \Gammabf \end{pmatrix}  \right ) \exp\left ( -\frac{\deltat}{2} {\bf A} \right ) \exp\left ( -\frac{\deltat}{2} {\bf B} \right ),
\end{equation*}
with 
\[
{\bf A} = \begin{pmatrix}
\0 & \0  &\0 \\
-\Omegabf & \0 & \0 \\
 \0 & \0 & \0 
\end{pmatrix},
\quad
{\bf B} = \begin{pmatrix}
\0 & \M^{-1}  &\0 \\
\0 & \0 & \0 \\
 \0 & \0 & \0 
\end{pmatrix},
\]
we see that $\Ed$ is indeed a symmetric splitting approximation of $\Ec=\exp(-\deltat \G)$. 
For sufficiently small $\deltat>0$ the corresponding Baker-Campbell-Hausdorff expansion converges so that
\[
\Ed = \exp(- \deltat (\G+ R_{2}(\deltat))
\]
where the remainder term is of second order in $\deltat$, i.e., $\norm{R_{2}(\deltat)}_{\C} \leq c \deltat^{2}$ for some $c>0$. Thus in particular
\[
\norm{\Ed}_{\mathcal{B}(\C)} \leq  \norm{\exp(- \deltat \G)}_{\mathcal{B}(\C)}  \exp( \deltat^{3}c) \leq \exp(-\deltat \lambda  + c \deltat^{3} ),
\]
which implies the statement for sufficiently small choice of $\deltat^{*}>0$.
\end{proof}
\begin{lemma}\label{lem:S:linear} There is a constant $c>0$ so that $\norm{\S \S^{\trans}}_{\mathcal{B}(\C)} \leq c h$.
\end{lemma}
\begin{proof}
Since $\Sw^{\trans} \Sw$ is symmetric and positive semi-definite it is by norm equivalence sufficient to show that the largest eigenvalue of this matrix is bounded from above by $c\deltat$. Recall that by definition  $\x^{\trans}\Sw^{\trans} \Sw \x= \z^{\trans}  [ \Qw -\F^{\trans}\Qw\F ]\z$ for all $\x = (\q,\z) \in \RR^{2n+m}$. Similarly as in the derivation of \eqref{eq:Ec:contract} we can use Gr\"onwall's inequality to find $\z^{\trans}\F^{\trans}\Qw\F\z \geq \exp(-\kappa \deltat) \z^{\trans} \Qw \z $ for all $\z \in \RR^{n+m}$ with  $\kappa = \sigma_{\max}[ \Gammabf^{\trans}\Qw + \Qw \Gammabf] / \sigma_{\min}(\Qw)$. Using an upper linear bound of the exponential function this then implies 
\[
\begin{aligned}
\x^{\trans}\Sw^{\trans} \Sw \x = \z^{\trans}\Qw\z -\z^{\trans}\F^{\trans}\Qw\F\z &\leq  (1- \exp(-\kappa \deltat)) \z^{\trans}\Qw\z
&\leq c\deltat \z^{\trans}\Qw\z
\end{aligned}
\]
for all $\x= (\q,\z)\in \RR^{2n+m}$, suitably chosen $c>0$ and sufficiently small $\deltat>0$. This concludes the proof. 
 \end{proof}

Equipped with the above two lemmas, we can now show existence of suitable constants $a_{l}>0,b_{l}\in \RR$ so that \eqref{eq:norm:Lya} is satisfied. 
We start with the case $l=1$. Let $u_{\max} :=\sup_{(\q,\q') \in \RR^{2n}}\norm{R(\q,\q') }_{\C} < \infty$, then
\begin{equation}\label{eq:K1}
\begin{aligned}
\hat{\Pc}^{{\textnormal{gle-BAOAB}}}_{\deltat} \norm{x}_{\C}^{2} &= \EE [ \norm{\numint^{\textnormal{gle-BAOAB}}(\x)}_{\C}^{2} ]\\
&\leq  \norm{\Ed \x}_{\C}^{2} + 2 \deltat  \norm{\Ed\x}_{\C} u_{\max}+ \deltat^{2}u_{\max} +  \EE \left [ \norm{ \Psi_{{\rm B}} \Psi_{{\rm A}} \widetilde{\S} \begin{pmatrix} \0 \\ \Rand_{k} \end{pmatrix}}_{\C}^{2}  \right ]\\
& \leq e^{-2\deltat \lambda' } \norm{\x}^{2}_{\C} + 2 \deltat e^{-\deltat \lambda' }u_{\max}  \norm{\x}_{\C}    + h^{2} u_{\sup}^{2} +\underbrace{ \norm{ \Psi_{{\rm B}} \Psi_{{\rm A}}}_{\mathcal{B}(\C)} \norm{ \Sw}_{\mathcal{B}(\C)}^{2} \EE [ \Rand^{\trans}\C\Rand]}_{ = O(\deltat) \text{ by \cref{lem:S:linear} } } \\
\end{aligned}
\end{equation}
where the last inequality follows from \cref{lem:E:contraction}. Using the inequality
\begin{equation}\label{eq:x:epsilon}
x^{2l-b} \leq \varepsilon x^{2l} + \frac{1}{\varepsilon^{2l}},
\end{equation}
which holds for any value of $x \in \RR$, any non-negative $\varepsilon<1$ and any non negative integer with $b<2l$, as well as linear bounds on the exponential function, we find 
\[
\begin{aligned}
e^{-2\deltat \lambda' } \norm{\x}^{2}_{\C} + 2 \deltat e^{-\deltat \lambda' }u_{\max}  \norm{\x}_{\C}  
& \leq  \left [ e^{-2\deltat \lambda' } +2\varepsilon \deltat e^{-\deltat \lambda' } u_{\max} \right ]  \norm{\x}^{2}_{\C}  +  \frac{2 \deltat e^{-\deltat \lambda' }u_{\max} }{\varepsilon}\\
& \leq e^{-\deltat \lambda' }  \left [ 1 - \frac{\lambda'}{2}\deltat +2\varepsilon u_{\max} \deltat  \right ]   \norm{\x}^{2}_{\C} + \frac{2 \deltat u_{\max} }{\varepsilon},\\
& \leq e^{-\deltat \lambda' } \norm{\x}^{2}_{\C} +  \deltat \frac{8 u_{\max}^{2}}{\lambda'}, 
\end{aligned}
\]
for all $\deltat \in [0,\deltat^{*})$ provided that $\deltat^{*}>0$ is chosen sufficiently small. The last equality follows by choosing $\varepsilon = \frac{\lambda'}{4 u_{\max}}$. Together with the observation that by \cref{lem:S:linear} the last summand in the last line of \eqref{eq:K1} is in $O(h)$, this shows that for $l=1$ the Lyapunov condition of \eqref{eq:norm:Lya} is satisfied with $a_{1} = \lambda'$ and $b_{1}$ chosen sufficiently large.

For $l>1$, we can write  
\[
\begin{aligned}
\hat{\Pc}^{{\textnormal{gle-BAOAB}}}_{\deltat} \norm{\x}_{\C}^{2l}& =  \norm{\Ed \x}_{\C}^{2l} +  \EE [ R_{l}(\Ed\x,\Rand) ] 
\end{aligned}
\]
where the expectation of the remainder term $R_{l}(\Ed\x,\Rand)$ can be bounded from above by $P_{\deltat} \left (\norm{\Ec\x}_{\C} \right)$, where $P_{\deltat}$ is a polynomial with degree of at most $2l-1$ and coefficients which are products of multiples and/or powers of $\norm{\Sw}_{\C}^{2}$ and $\deltat u_{\max}$, (and thus by \cref{lem:S:linear}   behave as $O(\deltat)$ as $\deltat \rightarrow 0$). In particular, by inequality \eqref{eq:x:epsilon} 
\[
\EE [ R_{l}(\Ed\x,\Rand) ]  \leq P_{\deltat}(\norm{\Ed\x}_{\C}) \leq  \deltat \varepsilon c_{1} \norm{\Ed\x}_{\C}^{2l} +  \deltat \varepsilon^{-2l}c_{2}
\]
for all $\varepsilon, 0<\varepsilon<1$, where $c_{1}, c_{2}$ are some $l$ and $\deltat^{*}$ dependent positive constants. Therefore,
\[
\begin{aligned}
\hat{\Pc}^{{\textnormal{gle-BAOAB}}}_{\deltat} \norm{\x}_{\C}^{2l} &\leq   \norm{\Ed \x}_{\C}^{2l} + \deltat \varepsilon c_{1} \norm{\Ed\x}_{\C}^{2l} +  \deltat \varepsilon^{-2l}c_{2}
&\leq e^{- 2\lambda' l \deltat} ( 1+ \deltat \varepsilon c_{1} ) +  \deltat \varepsilon^{-2l}c_{2}\\
&\leq e^{- ( 2\lambda' l - \varepsilon c_{1}) \deltat } +  \deltat \varepsilon^{-2l}c_{2},
\end{aligned}
\]
so that \eqref{eq:norm:Lya} is satisfied for $a_{l} = 2\lambda' l - \varepsilon c_{1}>0$ and $b_{l} = \varepsilon^{-2l}c_{2}$ provided $\varepsilon>0$  is sufficiently small. 

\subsection{Minorization condition}

Denote by 
\[
\Psi_{\deltat}(\x,\Rand_{1},\Rand_{2}) =  \Phi^{\rm BAOAB}_{\deltat}( \Phi^{\rm BAOAB}_{\deltat}(\x,\Rand_{1}) ,\Rand_{2}),
\]
the stochastic flow map of the twice iterated gle-BAOAB scheme. Here, we explicitly include the Gaussian random variables $\Rand_{1}, \Rand_{2}$ of the first gle-BAOAB step and second gle-BAOAB step, respectively. In particular, $\left(\Pc^{2}_{\deltat}\varphi \right ) (\x) = \EE \left [ \varphi  \left (\Psi_{\deltat}(\x,\Rand_{1},\Rand_{2}) \right )  \right]$. We deduce the validity of \cref{as:umc} by (i) showing the validity of a localized minorization condition via Lemma 6.3 in \cite{benaim2015qualitative} and (ii) a compactness argument. Define
\[
\Psi_{\deltat,\x} : \RR^{2(m+n)} \rightarrow \xDomain,\quad (t_{1},t_{2}) \mapsto \Psi_{\deltat}(\x,t_{1},t_{2}).
\]
In order for the conditions of Lemma 6.3 in \cite{benaim2015qualitative} to be satisfied it is sufficient to show that the Jacobian $D_{t}\Psi_{\deltat,\x}$ is invertible for any value of $\x$ in $\xDomain$. Since $D_{t}\Psi_{\deltat,\x}$ is also continuous this in particular implies by the inverse function theorem that $\Psi_{\deltat,\x}$ is surjective for any  $\x\in \xDomain$. In order show the invertibility of $D_{t}\Psi_{\deltat,\x}$, we first notice that 
\[
\Psi_{\deltat}(\x,\Rand_{1},\Rand_{2})  = \Phi^{\rm BA}_{\deltat/2} \circ \Phi^{\rm \hat{O}ABA\hat{O}}_{\deltat}( \Phi^{\rm AB}(\x),\Rand_{1},\Rand_{2}),
\]
where $\Phi^{\rm \hat{O}ABA\hat{O}}_{\deltat}$ is such that $\Phi^{\rm \hat{O}ABA\hat{O}}_{\deltat}(\x,\frac{1}{\sqrt{2}}\Rand_{1},\frac{1}{\sqrt{2}}\Rand_{2})$ corresponds to the flow map of \textup{gle-OBABO}. Now,
\[
D_{t}\Psi_{\deltat,\x}(t_{1},t_{2}) =
D_{\hat{\x}} \Phi^{\rm BA}_{\deltat/2}(\hat{\x}) D_{t}\Phi^{\rm OABAO}_{\deltat}( \Phi^{\rm AB}(\x),t_{1},t_{2} ) \in \RR^{(2n+m) \times 2(n+m)}
\]
where $\hat{\x} = \Phi^{\rm OABAO}_{\deltat}( \Phi^{\rm AB}(\x),t_{1},t_{2} )$.
The Differentials in this expression have the form
\[
D_{\hat{\x}} \Phi^{\rm BA}_{\deltat/2}(\hat{\x}) = \begin{pmatrix}
\I_{n} & \frac{h}{2}\M^{-1} & \0 \\
-\frac{h}{2} \nabla^{2} U(\hat{\q} +\frac{h}{2}\M^{-1}\hat{\p}) & \I_{n} -\frac{h^{2}}{4} \M^{-1}\nabla^{2}U(\hat{\q} +\frac{h}{2}\M^{-1}\hat{\p}) & \0\\
\0 & \0 &\I_{m}
\end{pmatrix},
\]
where $\nabla^{2} U$ denotes the Hessian of the potential function, and
\[
D_{t}\Phi^{\rm \hat{O}ABA\hat{O}}_{\deltat}(\x)  = 
\begin{pmatrix}
\nabla_{t_{1}}\proj_{q}  \Phi^{\rm \hat{O}ABA\hat{O}}_{\deltat}(\x) & \0\\
\nabla_{t_{1}}\proj_{z}  \Phi^{\rm \hat{O}ABA\hat{O}}_{\deltat}(\x) & \nabla_{t_{2}}\proj_{\z}  \Phi^{\rm \hat{O}ABA\hat{O}}_{\deltat}(\x)
\end{pmatrix}
=
\begin{pmatrix}
\nabla_{t_{1}}\proj_{q}  \Phi^{\rm \hat{O}ABA\hat{O}}_{\deltat}(\x) & \0\\
\nabla_{t_{1}}\proj_{z}  \Phi^{\rm \hat{O}ABA\hat{O}}_{\deltat}(\x) & \S
\end{pmatrix},
\]
where we suppressed the arguments $t_{1},t_{2}$, and denote by $\proj_{q} : \x\mapsto  \q$ and $ \proj_{z} : \x \mapsto \z $ the orthogonal projection operators onto the position components and $z$-components, respectively. Since the absolute values of the eigenvalues of $\nabla^{2} U(\q)$ are uniformly bounded from above by \cref{as:gle:2}, \ref{as:gle:it:4}, it follows that $D_{\hat{\x}} \Phi^{\rm BA}_{\deltat/2}(\hat{\x}) $ is invertible for any $\hat{\x}$ provided that the stepsize $\deltat>0$ is chosen sufficiently small. Likewise, observing that 
\[
D_{t_{1}}\proj_{\q}  \Phi^{\rm \hat{O}ABA\hat{O}}_{\deltat}(\x)   =  \deltat\I_{n} \M^{-1} \proj_{\q}\S -\frac{\deltat^{3}}{4} \nabla^{2} U(\q + \frac{\deltat}{2}\M^{-1} \proj_{\q}  \S t_{1}),
\]
for sufficiently small $\deltat$ 
\[
\nabla_{\Rand_{2}}\proj_{\z}  \Phi^{\rm \hat{O}ABA\hat{O}}_{\deltat}(\x)=  \S,
\]
gives us that  $D_{t}\Phi^{\rm \hat{O}ABA\hat{O}}_{\deltat}(\x)$  has rank $2n+m$ for any value of $\x$ provided that $\deltat$ is sufficiently small. In conclusion,  $D_{t}\Psi_{\deltat,\x}(t_{1},t_{2}) $ has full rank, $2n+m$ for any $\x\in \xDomain$. In particular, $\Psi_{\deltat,\x}$ is surjective for any  $\x\in \xDomain$ by the inverse function theorem. 

By Lemma 6.3 in \cite{benaim2015qualitative} we have that for any $\x,\hat{\x}\in \xDomain$, there exist open vicinities $J_{\x}$, $J_{\hat{\x}}$ of $\x$ and $\hat{\x}$, respectively, and a  constant $c_{\x,\hat{\x}}>0$ so that 
\begin{equation}\label{eq:inversion:1}
\forall \x'\in J_{\x}\quad  \forall \,\varphi \in \mathcal{C}_{0}(\xDomain,\RR)\qquad 
\Pc^{2}_{\deltat}( \mathbbm{1}_{J_{\hat{\x}}}\varphi ) (\x') 
> c_{\x,\hat{\x}} \int_{J_{\hat{\x}} } \varphi(\x) \lambda(\dd \x),
\end{equation} 
where $\lambda$ denotes the Lebesgue measure on $\xDomain$ and $\mathbbm{1}_{J}$ is the indicator function of the set $J$. 
Let $C\subseteq \xDomain$ be an arbitrary compact set in $\xDomain$. Compactness of $C$ implies that there is a finite set of pairs $\{ (\x_{i},\hat{\x}_{i})\}_{i=1}^{N} \subset C \times C$, with corresponding vicinities $J_{\x_{i}}, J_{\hat{\x}_{i}}$ satisfying \eqref{eq:inversion:1} so that the collection
\[
J_{\x_{i}} \times J_{\hat{\x}_{i}}, \quad i=1,\dots,N,
\] forms a cover of $C\times C$. If follows that for the choice $\nu = \lambda$ and $\alpha = \min_{i}c_{\x_{i},\hat{\x_{i}}}$ condition \cref{eq:as:umc} of \cref{as:umc} is guaranteed to be satisfied.

\section{Integration schemes for the underdamped and overdamped Langevin equation}
Here, we provide an algorithmic implementation of the \textnormal{ld-BAOAB} splitting method for the underdamped Langevin equation \eqref{eq:LD} which is referenced in \cref{col:discrete:WN}. 
\begin{algorithm}[H]
\caption{ ld-BAOAB}
\label{alg:ld:baoab}
\begin{algorithmic}
\setstretch{1.25}
\STATE {{\textbf{Input: }}{$(\q,\p)$}}
\STATE{$\p \gets  \p - \frac{\deltat}{2} \nabla U(\q)$}
\STATE{$\q \gets \q +  \frac{\deltat}{2}\M^{-1}\p$}
\STATE{$\p \gets \widehat{{\bm F}}_{\deltat} \p + \widehat{{\bm S}}_{\deltat}\widehat{\Rand}$}
\STATE{$\q \gets \q +  \frac{\deltat}{2}\M^{-1}\p$}
\STATE{$\p \gets  \p -  \frac{\deltat}{2}\nabla U(\q)$}
\STATE{\textbf{Output: }{$(\q,\p)$}}
\end{algorithmic}
\end{algorithm}
In this algorithm $\widehat{{\bm F}}_{\deltat} =\exp(-\deltat \GammaLD\M^{-1})$ denotes the Matrix exponential of $ -\deltat\GammaLD\M^{-1}$, $\widehat{{\bm S}}_{\deltat}\widehat{{\bm S}}_{\deltat}^{\trans} = \beta^{-1} [ \M  - \widehat{{\bm F}}_{\deltat} \M\widehat{{\bm F}}^{\trans}_{\deltat}]$, and $\widehat{\Rand}\sim \mathcal{N}(\0,I_{n})$. In particular, for $\Gammabf = \gamma \I_{n},\gamma>0$, and $\M=\diag(m_{1},\dots,m_{n}), m_{i}>0, 1\leq i \leq n$, the scheme reduces to the BAOAB splitting method for the underdamped Langevin equation first introduced in \cite{Leimkuhler2013a}.

\section{ Limit of $2 \times 2$-matrix exponential}\label{supp:sec:limit:22exp}
The following Lemma is used in the main text in the derivation of the white noise limit of the \textnormal{gle-BAOAB} method (\cref{col:discrete:WN}).
\begin{lemma}\label{lem:exp:limit} For any $a \in \RR, \, b>0$ we have
\[
\lim_{\epsilon \rightarrow 0} \exp\left (  -\begin{pmatrix} 0 & -\epsilon^{-1}a \\ \epsilon^{-1}a & \epsilon^{-2}b \end{pmatrix} \right )  =\begin{pmatrix} e^{-a^{2}/b} & 0 \\ 0 & 0 \end{pmatrix}.
\]
\end{lemma}
\begin{proof}
Define the shorthand
\[
A^{\epsilon} = \begin{pmatrix} 0 & -\epsilon^{-1}a \\ \epsilon^{-1}a & \epsilon^{-2}b \end{pmatrix}
\]
and $c_{\epsilon} = \sqrt{b^2-4 a^2 \epsilon ^2}$. The eigenvalues of the matrix $A^{\epsilon}$ are
\[
\lambda_{1}^{\epsilon} =\frac{c_{\epsilon}+b}{2 \epsilon ^2},~
\lambda_{2}^{\epsilon} =\frac{-c_{\epsilon}+b}{2 \epsilon ^2},
\]
with corresponding eigenvectors 
\[
v_{1} = \begin{pmatrix} -\frac{b-c_{\epsilon}}{2 a \epsilon } \\ 1 \end{pmatrix},~
v_{2} = \begin{pmatrix} -\frac{c_{\epsilon}+b}{2 a \epsilon } \\ 1 \end{pmatrix},
\]
respectively, thus
\[
\begin{aligned}
\exp(A^{\epsilon} ) &= [v_{1},v_{2}]  \diag(e^{-\lambda_{1}^{\epsilon}},e^{-\lambda_{2}^{\epsilon}}) [v_{1},v_{2}] ^{-1} \\
&=
\begin{pmatrix}
e^{-\lambda_{2}^{\epsilon}}   \frac{\left(b+c_{\epsilon}\right)}{2 c_{\epsilon}}-e^{-\lambda_{1}^{\epsilon}} \frac{ \left(b-c_{\epsilon}\right)}{2 c_{\epsilon}} 
 &
 \left ( e^{-\lambda_{2}^{\epsilon}} -e^{-\lambda_{1}^{\epsilon}}  \right ) \frac{a \epsilon }{c_{\epsilon}}
 \\
-\left ( e^{-\lambda_{2}^{\epsilon}} -e^{-\lambda_{1}^{\epsilon}}  \right ) \frac{a \epsilon }{c_{\epsilon}}
  & \frac{e^{-\lambda_{1}^{\epsilon}} \left(b+c_{\epsilon}\right)}{2 c_{\epsilon}}-\frac{e^{-\lambda_{2}^{\epsilon}} \left(b-c_{\epsilon}\right)}{2 c_{\epsilon}}
\end{pmatrix},
\end{aligned}
\]
and the result follows since
\[
\lambda_{1}^{\epsilon} \rightarrow \infty, ~\lambda_{2}^{\epsilon} \rightarrow  \frac{a^{2}}{b}, ~c_{\epsilon} \rightarrow b,
\]
as $\epsilon \rightarrow 0$.
\end{proof}

\section{Derivation of leading order term in solution of \eqref{eq:exp:epsilon}}\label{supp:sec:pde:solution}
In what follows we derive the form of the leading order term $f_{3,0}$ from the collection of PDEs collection of PDEs \cref{pde1b,pde2b,pde3b,pde4b}.

A simple calculation shows that $\Lco^{*} \left (  -\frac{1}{8 }\beta  \p^2 U''(\q)  \right ) $
coincides with the right hand side of \cref{pde1b}, thus 
 \begin{equation}\label{eq:f20:gen}
 f_{3,0}(\q,\p,\s) = -\frac{1}{8 }\beta  \p^2 U''(\q) + \phi(\q),
\end{equation}
and this solution of \cref{pde1b} is uniquely determined up to the term $\phi(\q)$. This follows since by ergodicity of the OU-process associated with $\Lco^{*}$, the null-space of $\Lco^{*}$ consists of functions which are constant in $(\p,\s)$.
Substituting this form of $f_{3,0}$ in \eqref{pde2b} yields
\begin{equation}\label{eq:f21}
\Lco^{*}f_{3,1}(\q,\p,\s)= \frac{1}{24} \beta  \p^3 U^{(3)}(\q)+\p \phi'(\q).
\end{equation}
Again, by the same arguments as used to derive the generic form of $f_{3,0}$, we find that the solution $f_{3,1}$ of \cref{eq:f21} is uniquely determined up to a function $\Psi(\q)$, i.e.,
\begin{equation}\label{eq:f21:form}
f_{3,1}(\q,\p,\s) = g_{\Gammabf}(\q,\p,\s) + \Psi(\q),
\end{equation}
where 
\begin{equation*}
\begin{aligned}
&g_{\Gammabf}(q,p,s)=\\
&\frac{1}{g(\Gammabf)}  \Big[ -2 \left(s  \Gamma _{2,1}-p \Gamma _{2,2}\right) \left(\Gamma _{2,2}^2
   \left(\left(\beta  p^2+6\right) U^{(3)}(q)-72 \phi'(q)\right)-2 \beta  s  p \Gamma
   _{2,2} \Gamma _{2,1} U^{(3)}(q)+\Gamma _{2,1}^2 \left(\beta  s ^2-3\right)
   U^{(3)}(q)\right)\\
   +& 2 \Gamma _{1,1}^2 \left(p \Gamma _{2,2} \left(\left(\beta 
   p^2+6\right) U^{(3)}(q)-72 \phi'(q)\right)+9 s  \Gamma _{2,1} \left(8
   \phi'(q)-U^{(3)}(q)\right)\right)\\
   +&\Gamma _{1,1} \left(\Gamma _{2,1} \left(3 s  \Gamma
   _{2,2} \left(120 \phi'(q)-\left(\beta  p^2+14\right) U^{(3)}(q)\right)-2 p \Gamma _{1,2}
   \left(\beta  p^2-3\right) U^{(3)}(q)\right)\right)\\
   +& \Gamma _{1,1}  5 p \Gamma _{2,2}^2 \left(\left(\beta 
   p^2+6\right)U^{(3)}(q)-72 \phi'(q)\right)\\
   +&\Gamma _{1,2} \Gamma _{2,1} \left(3
   s  \Gamma _{2,1} \left(\left(\beta  p^2+2\right) U^{(3)}(q)-24 \phi'(q)\right)+p
   \Gamma _{2,2} \left(\left(6-5 \beta  p^2\right) U^{(3)}(q)+72 \phi'(q)\right)\right) \Big  ]  
\end{aligned}
\end{equation*}
with
\[
\small
g(\Gammabf) = 72
   \left(\Gamma _{1,2} \Gamma _{2,1}-\Gamma _{1,1} \Gamma _{2,2}\right) \left(-2 \Gamma
   _{1,1}^2-5 \Gamma _{2,2} \Gamma _{1,1}-2 \Gamma _{2,2}^2+\Gamma _{1,2} \Gamma
   _{2,1}\right).
\]
Next, by virtue of the Fredholm alternative, there exists a solution for \cref{pde3b} if and only if $\innerLmu{g, \Lch^{*}f_{3,1}} = 0$
for all functions $g$ contained in the null-space of $\Lco$. This is exactly the case if 
\begin{equation}\label{eq:final:eq:sdsds}
\begin{aligned}
0& =\int_{\pDomain \times \sDomain}\mathcal{L}^{*}_{H}f_{3,1}(\q,\p,\s) \muinv( \dd \p \,\dd \s) \\
&=
\frac{\Gamma _{2,2}}{\Gamma _{1,1} \Gamma _{2,2}-\Gamma _{1,2} \Gamma _{2,1}} \Big ( 
\frac{U^{(4)}(\q)}{8 \beta }+U'(\q) \phi'(\q)
-\frac{1}{8} U^{(3)}(\q) U'(\q)-\frac{\phi''(\q)}{\beta} \Big ). 
\end{aligned}
\end{equation}
for $\muinv$-almost all $\q \in \RR^{n}$. The right hand side of \cref{eq:final:eq:sdsds} vanishes if and only if $\phi(\q) = \frac{1}{8}U''(\q)$, and we conclude
\[
 f_{3,0}(\q,\p,\s) = -\frac{1}{8 }\beta  \p^2 U''(\q) +  \frac{1}{8}U''(\q).
\]

\section{Details on numerical simulations}
\subsection{Specification of Bayesian Gaussian mixture model for Hidalgo stamps dataset}\label{supp:sec:hidalgo}
We assume a likelihood of the form 
\[
 p(y_{i},z_{i})_{1\leq i \leq N} \given \q )  = \sum_{k=1}^{N_{c}} {\bm w}_{k} {\bm \lambda}_{k}^{1/2} \exp \left (-\frac{{\bm \lambda}_{k}}{2}(y_{i}-{\bm \mu}_{k} )^{2}\right ) 
\]
and a hierarchical prior specified by
\[ 
\begin{aligned}
{\bf \mu}_{k} &\sim \mathcal{N}(m, \kappa^{-1}),\\
{\bf \lambda}_{k} & \sim \text{Gamma}(\alpha,\beta),\\
\beta & \sim \text{Gamma}(g,h),\\
({\bf w}_{1},{\bf w}_{2},{\bf w}_{3}) &\sim \text{Dirichlet}_{3}(1,1,1).
\end{aligned}
\]
with $m = M, \kappa = 4/R^{2},\,\alpha=2,\,g=0.2,\,h=100g/(\alpha R^{2})$, where $M$ and $R$ denotes the mean and the range of the data $(y_{i})_{1\leq i\leq N}$, respectively.

 Under these assumptions the resulting posterior density reads (for a derivation see \cite[Section 2.1]{Stoltz2012}):
 \[
 \begin{aligned}\label{eq:dens:bayes:mixture}
 p(\thetabf \given (y_{i},z_{i})_{1\leq i \leq N} ) =\,&  \frac{\kappa^{K/2}g^{h}\beta^{N_{c} \alpha + g -1}}{\Gamma(\alpha)^{K}\Gamma(g)(2\pi)^{\frac{n+K}{2}}}\left ( \prod_{k=1}^{N_{c}}{\bm \lambda}_{k}\right )^{\alpha-1}\\
& \times \exp\left ( -\frac{\kappa}{2}\sum_{k=1}^{N_{c}}({\bm \mu}_{k}-M)^{2}-\beta \left ( h+ \sum_{k=1}^{K}{\bm \lambda}_{k}\right ) \right )\\
& \times \prod_{i=1}^{N} \left [ \,\sum_{k=1}^{N_{c}} {\bm w}_{k} {\bm \lambda}_{k}^{1/2} \exp \left (-\frac{{\bm \lambda}_{k}}{2}(y_{i}-{\bm \mu}_{k} )^{2}\right ) \right ],
 \end{aligned}
 \]
where $\Gamma(\cdot)$ denotes the gamma function.

\subsection{Parameterization of pre-optimized memory kernel}\label{sec:opt:memory:kernel}
The matrix representation of the memory kernel used in the numerical experiment described in \cref{sec:Hidalgo} is of the form
\[
\Gammabf_{\text{kv\_8\_8}} := \I_{n} \otimes \begin{pmatrix} \Gammabf_{1,1} & \Gammabf_{1,2} \\ \Gammabf_{2,1} & \Gammabf_{2,2} \end{pmatrix},
\]
where $\otimes$ denotes the standard Kronecker product, $n=9$,  and 
\[
\begin{aligned}
\Gammabf_{1,1} &= (\begin{smallmatrix}1.336001\text{E+}1\end{smallmatrix}),\\
\Gammabf_{1,2} &=  (\begin{smallmatrix} 8.327012\text{E-}6&  1.850437\text{E-}4&  2.551111\text{E-}3&  -1.63314\text{E-}2 -1.334317\text{E-}1&  1.679873\text{E+}0&  2.22050\text{E+}1&  6.274743\text{E+}0\end{smallmatrix}),\\
\Gammabf_{2,1} &=  (\begin{smallmatrix} 1.20137\text{E-}5& 1.83376\text{E-}4& 2.505216\text{E-}3& -1.625490\text{E-}2& -1.334317 & 1.68179\text{E+}0& 2.22086\text{E+}1& 6.09239\text{E+}0\end{smallmatrix})^{\trans},\\
 \Gammabf_{2,2} &=  
  {\fontsize{.2cm}{.3cm}
   \left (
 \begin{smallmatrix}
  3.25571\text{E-}6 & 7.47982\text{E-}6&  9.622039\text{E-}6&  4.244713\text{E-}5&
            1.695553\text{E-}5&  3.285529\text{E-}5&  6.747855\text{E-}6&  -1.438594\text{E-}4\\
          -7.479820\text{E-}6&  1.019983\text{E-}4&  1.424663\text{E-}5&  -1.396013\text{E-}5&  1.513710\text{E-}5&  -1.200000\text{E-}5&  2.54657\text{E-}5&  -6.346355\text{E-}5\\
          -9.622039\text{E-}6&  -1.424663\text{E-}5&  2.206513\text{E-}3&  2.2645018\text{E-}5&  1.384732\text{E-}5&  4.388201\text{E-}5&  -4.079418\text{E-}6&  8.9663528\text{E-}3\\
         -4.244713\text{E-}5&  1.396013\text{E-}5&  -2.26450\text{E-}5&  2.218067\text{E-}2&  1.249881\text{E-}5&  2.492691\text{E-}5&  1.116974336243\text{E-}5&  3.14859310\text{E-}3\\
        -1.6955539\text{E-}5&  -1.513710\text{E-}5&  -1.3847329\text{E-}5&  -1.249881\text{E-}5&  1.772222\text{E-}1&  3.888513\text{E-}5&  -4.4198267\text{E-}6&  1.010943\text{E-}1\\
         -3.285529\text{E-}5&  1.200000\text{E-}5&  -4.388201\text{E-}5&  -2.492691\text{E-}5&  -3.888513\text{E-}5&  2.79177\text{E+}0&  8.375329164851\text{E-}6&  2.17612\text{E-}1\\
         -6.74785\text{E-}6&  -2.54657\text{E-}5&  4.07941\text{E-}6&  -1.116974\text{E-}5&  4.419826\text{E-}6&  -8.37532\text{E-}6&  4.043272\text{E+}1&  3.460867415178\text{E-}2\\
        1.438594\text{E-}4&  6.346355\text{E-}5&  -8.966352\text{E-}3&  -3.14859\text{E-}3&  -1.010943\text{E-}1&  -2.176129826302\text{E-}1&  -3.46086\text{E-}2&  1.088614\text{E+}3
        \end{smallmatrix}
        \right ).
        }
\end{aligned}
\]

\subsection{Splitting methods based on alternative decompositions}\label{sec:construction:stochastic-verlet}
For the sake of completeness and a self-contained presentation, we include a brief discussion of alternative splitting schemes for the GLE, and some previously proposed schemes. 
\subsubsection{The \textnormal{gle-ASA} and \textnormal{gle-SAS} methods}
From the decomposition of the Markovian reformulation of the GLE as 
\begin{equation}\label{splitting:GLE:1}
\begin{bmatrix}
\dd \q\\
\dd \p\\
\dd \s
\end{bmatrix}
=
\underbrace{
\begin{bmatrix}
\p \dd t\\
\0\\
\0
\end{bmatrix}
}
_\textrm{{=:A}}
+
\underbrace{
\begin{bmatrix}
\0\\
\begin{pmatrix}
- \nabla U(\q) \\
\0
\end{pmatrix}\dd t
-\Gammabf 
\begin{pmatrix}
\p\\
\s
\end{pmatrix} \dd t
+\Sigmabf \dd \W_{t}
\end{bmatrix}
}
_\textrm{{=:S}},
\end{equation}
we can construct the symmetric stochastic splitting
\[
\hat{\Phi}^{{\textrm{gle-ASA}}}_{\deltat} = \Phi^{A}_{\deltat/2} \circ \Phi^{S}_{\deltat} \circ \Phi^{A}_{\deltat/2}, \quad \hat{\Phi}^{{\textrm{gle-SAS}}}_{\deltat} = \Phi^{S}_{\deltat/2} \circ \Phi^{A}_{\deltat} \circ \Phi^{S}_{\deltat/2},
\]
where 
\begin{equation*}
\Phi^{S}_{\deltat}: (\q,\p,\s) \mapsto  \big (\q, \F\; (\p, \s)^{\trans} + \S \Rand  + \Gammabf^{-1}[\I_{n+m} - \F] (-\nabla U(\q), \0 )^{\trans} \big ), \; \Rand \sim \mathcal{N}( \0, \I_{n+m}),
\end{equation*}
and $\Phi^{A}_{\deltat}$ and $ \F$ and $\S$ are all as specified in \cref{sub:splitting:implentation}.
These methods resemble the stochastic position Verlet (SPV) and the stochastic velocity Verlet (SVV) methods which have been proposed in  \cite{Melchionna2007} as integration schemes for the white noise Langevin  equation.

\subsubsection{The \textnormal{BB-BAOB} and the  \textnormal{BB-BACOCAB} methods}
In \cite{Baczewski2013a} the authors propose a family of numerical integrators based on an extended variable formalism specifically designed for memory kernels, which take the form of a Prony series and vanishing cross-correlation terms, 
\[
\K_{ij}(t) = 
\begin{cases}
0 &\text{ if } i \neq j,\\
 \sum_{k=1}^{m} \frac{c_{k}}{\tau_{k}} e^{-|t|/\tau_{k}} &\text{ if } i = 0.
\end{cases}
\]
which in terms of \eqref{eq:qgle} corresponds to a choice of $\Gammabf$ as,
\begin{equation}
\begin{aligned}
\Gammabf_{1,1} &= \0,  &\Gammabf_{2,2} = \I_{n} \otimes \diag({1/\tau_{1}, \dots, 1/\tau_{m}}),\\
\Gammabf_{2,1} &=\Gammabf_{1,2}^{\trans}, &\Gammabf_{1,2} = \I_{n} \otimes (\sqrt{c_{1}/\tau_{1}}, \dots, \sqrt{c_{m}/\tau_{m}})
\end{aligned}
\end{equation}
Multiple splitting schemes are proposed in this work. The method to which we refer as \textnormal{BB-BAOB} is based on a splitting of the form
\begin{equation}\label{eq:bb:decomp:1}
\begin{bmatrix}
\dd \q\\
\dd \p\\
\dd \s\\
\end{bmatrix}
=
\underbrace{
\begin{bmatrix}
\p \dd t\\
\0\\
\0
\end{bmatrix}
}
_\textrm{{A}}
+
\underbrace{
\begin{bmatrix}
\0\\
- \nabla U(\q) - \Gammabf_{1,2} \s \dd t\\
\0\\
\end{bmatrix}
}
_\textrm{{B}}
+
\underbrace{
\begin{bmatrix}
\0\\
\0\\
\Gammabf_{2,1} \p - \Gammabf_{2,2}\s \dd t+\Sigmabf_{2,2} \dd \W_{t}
\end{bmatrix}
}
_\textrm{{O}}
\end{equation}
The method to which we refer as {\textnormal{BB-BACOCAB}} is based on a splitting of the form 
\begin{equation}\label{eq:bb:decomp:2}
\begin{bmatrix}
\dd \q\\
\dd \p\\
\dd \s\\
\end{bmatrix}
=
\underbrace{
\begin{bmatrix}
\p \dd t\\
\0\\
\0
\end{bmatrix}
}
_\textrm{{A}}
+
\underbrace{
\begin{bmatrix}
\0\\
-\nabla U(\q)\dd t\\
\0\\
\end{bmatrix}
}
_\textrm{{B}}
+
\underbrace{
\begin{bmatrix}
\0\\
\Gammabf_{1,2} \s \dd t\\
\0\\
\end{bmatrix}
}
_\textrm{{C}}
+
\underbrace{
\begin{bmatrix}
\0\\
\0\\
\Gammabf_{2,1}\p - \Gammabf_{2,2}\s \dd t+\Sigmabf_{2,2} \dd \W_{t}
\end{bmatrix}
}
_\textrm{{O}}
\end{equation}
It is easy to see that in the white noise limit, i.e., as $\tau_{k}\rightarrow \infty$, the (exact) Euler updates corresponding to the solutions of the \textnormal{B}-part in \eqref{eq:bb:decomp:1} and the \textnormal{C}-part in \eqref{eq:bb:decomp:2}, respectively, become unstable. 
For this reason the authors construct  the scheme {\textnormal{BB-BAOB}}  as 
\begin{equation}\label{Alg:gle-BB3}
\begin{aligned}
\p_{i}^{n+1/2} &= \p_{i}^{n} - \frac{\deltat}{2}\nabla_{\q_{i}}U(\q_{i}^{n}) + \frac{\deltat}{2}\sum_{k=1}^{m}\s_{k}^{n}\\
\q_{i}^{n+1}  &=\q_{i}^{n} + \deltat m^{-1} \p_{i}^{n+1/2}\\
\s_{k}^{n+1} &= \theta_{k} \s_{k}^{n} - (1-\theta_{k})c_{k} \p_{i}^{n+1/2} + \alpha_{k}\sqrt{2\beta^{-1}c_{k}}\Rand_{k}^{n}\\
\p_{i}^{n+1} &= \p_{i}^{n+1/2} - \frac{\deltat}{2}\nabla_{\q_{i}}U(\q_{i}^{n+1})+ \frac{\deltat}{2}\sum_{k=1}^{m}\s_{k}^{n+1} \\
\end{aligned}
\end{equation}
and the scheme {\textnormal{BB-BACOCAB}} as
\begin{equation}\label{Alg:gle-BB3b}
\begin{aligned}
\p_{i}^{n+1/2} &= \p_{i}^{n} - \frac{\deltat}{2} \nabla_{\q_{i}}U(\q^{n})\\
\q_{i}^{n+1/2} &= \q_{i}^{n}+ m^{-1}\frac{\deltat}{2} \p_{i}^{n} \\
\tilde{\p}^{n+1/2} &= \p_{i}^{n+1/2} - \frac{\deltat}{2} \sum_{i=1}^{m}\s_{i}^{n} \\
\s_{k}^{n+1}
& =\theta_{k}\s_{k}^{n} - (1- \theta_{k})c_{k}m^{-1} \tilde{\p}^{n+1/2} + \alpha_{k}\sqrt{2\beta^{-1} c_{k}}\Rand_{k}^{n}\\
\hat{\p}^{n+1/2} &= \tilde{\p}^{n+1/2} - \frac{\deltat}{2} \sum_{i=1}^{m}\s_{i}^{n+1} \\
\q_{i}^{n+1} &=\q_{i}^{n+1/2} + \frac{\deltat}{2} \hat{p}^{n+1/2}\\
\p_{i}^{n+1} &=\hat{p}^{n+1/2} - m^{-1} \frac{\deltat}{2} \nabla_{\q_{i}}U(\q^{n+1/2})
\end{aligned}
\end{equation}
where  $\Rand_{k}^{n}, ~~ 1 \leq k \leq m, ~ n \in \mathbb{N}$ are i.i.d. normal distributed random variables.
 and the coefficients $\theta_{k},\alpha_{k}$ are chosen as $\theta_{k} = e^{-\deltat/\tau_{k}}$, and 
 \begin{equation}\label{eq:alpha}
 \alpha_{k} = \sqrt{ \frac{(1-\theta_{k})^{2}}{\deltat}}.
\end{equation}
 
 These schemes resemble splitting schemes in the sense that if the $\alpha_{k}$s were instead set to $\alpha_{k}=\sqrt{ {(1-\theta_{k})^{2}}}$, then, 
 the above schemes would exactly correspond to splitting schemes corresponding to the integration sequences BAOAB and BACOCAB, respectively.
 By choosing $\alpha_{k}$ instead as specified in \eqref{eq:alpha} the scheme remains stable in the white-noise limit.

\end{document}